\documentclass[12pt]{amsart}
\usepackage{amsmath,amssymb, amscd}
\usepackage{youngtab}
\usepackage[dvipdfmx]{graphicx}

\makeatletter
\@addtoreset{equation}{section}

\makeatother
\usepackage{enumerate}
\usepackage{color}
\usepackage[all]{xy}

\usepackage{color}

\def\bA{{\mathbb A}}

\def\CC{{\mathbb C}}

\def\GG{{\mathbb G}}

\def\RR{{\mathbb R}}
\def\ZZ{{\mathbb Z}}

\def\NN{{\mathbb N}}
\def\PP{{\mathbb P}}

\def\bH{{\mathbf F}}

\def\bH{{\mathbf H}}

\def\bs{{\mathbf s}}

\def\ii{{\sqrt{-1}}}

\def\bs{{\mathbf s}}

\def\cA{{\mathcal A}}

\def\cO{{\mathcal O}}

\def\cQ{{\mathcal Q}}

\def\fB{{\mathfrak B}}
\def\fC{{\mathfrak C}}

\def\fE{{\mathfrak E}}

\def\fK{{\mathfrak K}}

\def\fZ{{\mathfrak Z}}
\def\fa{{\mathfrak a}}
\def\fb{{\mathfrak b}}
\def\fc{{\mathfrak c}}

\def\fe{{\mathfrak e}}

\def\fh{{\mathfrak h}}

\def\fk{{\mathfrak k}}

\def\fm{{\mathfrak m}}

\def\fr{{\mathfrak r}}
\def\fs{{\mathfrak s}}

\def\fu{{\mathfrak u}}

\def\fy{{\mathfrak y}}

\def\rw{{\mathrm w}}

\def\tX{{\widetilde X}}

\def\tlh{{\widetilde h}}

\def\tell{{\widetilde \ell}}

\def\tp{{\widetilde p}}

\def\tw{{\widetilde w}}
\def\tell{{\widetilde \ell}}

\def\tdelta{{\widetilde \delta}}

\def\tvarphi{{\widetilde \varphi}}

\def\tfE{{\widetilde{\mathfrak E}}}

\def\tfe{{\widetilde{\mathfrak e}}}

\def\BH{{\overline H}}

\def\Bfr{{\overline{\mathfrak r}}}

\def\hc{{\widehat c}}

\def\hH{{\widehat H}}

\def\hN{{\widehat N}}

\def\hR{{\widehat R}}
\def\hS{{\widehat S}}

\def\hfy{{\widehat{\mathfrak y}}}

\def\hphi{{\widehat \phi}}

\def\hzeta{{\widehat \zeta}}

\def\hUpsilon{{\widehat \Upsilon}}

\def\rfy{{\mathring{\mathfrak y}}}

\def\lH{{\overline H}}

\def\Sp{\mathrm{Sp}}

\def\deg{\mathrm{deg}}

\def\Div{\mathrm{div}}

\def\Aut{\mathrm{Aut}}
\def\Gal{\mathrm{Gal}}

\def\Spec{\mathrm{Spec}\,}  

\def\wt{{\mathrm {wt}}}

\def\omegap#1{{\omega^{{\prime}}_{#1}}}
\def\omegapp#1{{\omega^{{\prime\prime}}_{#1}}}
\def\etap#1{{\eta^{{\prime}}_{#1}}}
\def\etapp#1{{\eta^{{\prime\prime}}_{#1}}}

\def\nuI#1{{\nu^{\mathrm{I}}_{#1}}}

\def\tnuII#1{{\widetilde\nu^{\mathrm{II}}_{#1}}}
\def\nuII#1{{\nu^{\mathrm{II}}_{#1}}}

\def\nuIII#1{{\nu^{\mathrm{III}}_{#1}}}
\def\nuIIIo#1{{\nu^{\mathrm{III}\circ}_{#1}}}

\def\trp{{\, {}^t\!}}

\def\book#1{\rm{#1}, }
\def\paper#1{\textit{#1}, }
\def\jour#1{\rm{#1}, }
\def\yr#1{({\rm{#1}) }}
\def\vol#1{\textbf{#1}}
\def\pages#1{\rm{#1}}

\def\publ#1{\rm{#1}, }
\def\by#1{{\rm{#1}, }}

\newtheorem{theorem}{Theorem}[section]
\newtheorem{definition}[theorem]{Definition}

\newtheorem{proposition}[theorem]{Proposition}
\newtheorem{corollary}[theorem]{Corollary}
\newtheorem{remark}[theorem]{Remark}
\newtheorem{lemma}[theorem]{Lemma}

\def\dfrac#1#2{{\displaystyle\frac{#1}{#2}}}

\def\book#1{\rm{#1}, }
\def\paper#1{\textit{#1}, }
\def\jour#1{\rm{#1}, }
\def\yr#1{({\rm{#1}) }}
\def\vol#1{\textbf{#1}}
\def\pages#1{\rm{#1}}

\def\publ#1{\rm{#1}, }
\def\by#1{{\rm{#1}, }}

\begin{document}

\title{Algebraic construction of the sigma function for general Weierstrass curves}

\author{J. Komeda, S. Matsutani and E. Previato}

\date{\today}

\begin{abstract}
The Weierstrass curve $X$ is a smooth algebraic curve determined by the Weierstrass canonical form, $y^r + A_{1}(x) y^{r-1} + A_{2}(x) y^{r-2} +\cdots + A_{r-1}(x) y  + A_{r}(x)=0$, where $r$ is a positive integer, and each $A_j$ is a polynomial in $x$ with a certain degree.
It is known that every compact Riemann surface has a Weierstrass curve $X$, which is birational to the surface.
The form provides the projection $\varpi_r : X \to \PP$ as a covering space.
Let $R_X := \bH^0(X, \cO_X(*\infty))$ and $R_\PP := \bH^0(\PP, \cO_\PP(*\infty))$.
Recently we have the explicit description of the complementary module $R_X^\fc$ of $R_\PP$-module $R_X$, which leads the explicit expressions of the holomorphic one form except $\infty$, $\bH^0(\PP, \cA_\PP(*\infty))$ and the trace operator $p_X$ such that $p_X(P, Q)=\delta_{P,Q}$ for $\varpi_r(P)=\varpi_r(Q)$ for $P, Q \in X\setminus\{\infty\}$.
In terms of them, we express the fundamental 2-form of the second kind $\Omega$ and its connection to the sigma function for $X$.
\end{abstract}




\maketitle

\section{Introduction}\label{introduction}

In Weierstrass' elliptic function theory, the algebraic properties associated with an elliptic curve of Weierstrass' standard equation $y^2 = 4 x^3 - g_2 x - g_3$ are connected with the transcendental properties defined on its Jacobi variety via the $\sigma$ function since $\Bigr(\wp(u)=-\dfrac{d^2}{d u^2}\log\sigma(u), $  $\dfrac{d\wp(u)}{d u}\Bigr)$ is identical to a point $(x,y)$ in the curve \cite{WeiWV,WhittakerWatson}. These algebraic and transcendental properties are equivalently obtained by the identity, and play the central role in the elliptic function theory.
Via the equivalence, the elliptic function theory affects several fields in mathematics, science, and technology.
In other words, in Weierstrass' elliptic function theory, the equivalence between the algebraic objects of the curve and the transcendental objects on its Jacobi variety is crucial.

Weierstrass himself extended the picture to general algebraic curves \cite{WeiWIV}, but it failed due to difficulties.
Some of the purposes in mathematics in the XX-th century were to overcome the difficulties which were achieved.
We have studied the generalization of this picture to algebraic curves with higher genera in the series of the studies \cite{KMP13, KMP16, KMP19} following Mumford's studies for the hyperelliptic curves based on the modern algebraic geometry \cite{Mum81, Mum84}, i.e., the unification of the theory of algebraic curves in the XIX-th century with the modern one.

The elliptic theta function was generalized by Riemann for an Abelian variety. 
In contrast, its equivalent function Al was defined for any hyperelliptic curve by Weierstrass, which was refined by Klein using only the data of the hyperelliptic Riemann surface and Jacobi variety as a generalization of the elliptic sigma function \cite{Klein}. 
Baker re-constructed Klein's sigma functions by using the data of hyperelliptic curves algebraically \cite{Baker97}.
Buchstaber, Enolskii, and Leykin extend the sigma functions to certain plane curves, so-called $(n,s)$ curves, based on Baker's construction which we call the EEL construction due to work by Eilbeck, Enolskii, and Leykin \cite{EEL00} [\cite{BEL20} and its references.]. 
For the $(n,s)$ curves with the cyclic symmetry, the direct relations between the affine rings and the sigma functions were obtained as the Jacobi inversion formulae \cite{MP08, MP14}. Additionally, we generalized the sigma functions and the formulae to a particular class of the space curves using the EEL-construction \cite{MK13, KMP13, KMP19}.

We have studied further generalization of the picture in terms of the Weierstrass canonical form \cite{KM2020, KMP2022a}.
The Weierstrass curve $X$ is a normalized curve of the curve given by the Weierstrass canonical form, $y^r + A_{1}(x) y^{r-1} + A_{2}(x) y^{r-2} +\cdots + A_{r-1}(x) y  + A_{r}(x)=0$, where $r$ is a positive integer and each $A_j$ is a polynomial in $x$ of a certain degree (c.f. Proposition \ref{2pr:WCF}) so that the Weierstrass non-gap sequence at $\infty\in X$ is given by the numerical semigroup $H_X$ whose generator contains $r$ as its minimal element.
It is known that every compact Riemann surface has a Weierstrass curve $X$, which is birational to the surface.
We also simply call the Weierstrass curve \emph{W-curve}.

It provides the projection $\varpi_r : X \to \PP$ as a covering space.
Let $R_X := \bH^0(X, \cO_X(*\infty))$ and $R_\PP := \bH^0(\PP, \cO_\PP(*\infty))$.
In \cite{KMP2022a}, we have the explicit description of the complementary module $R_X^\fc$ of $R_\PP$-module $R_X$, which leads the explicit expressions of the holomorphic one form except $\infty$, $\bH^0(\PP, \cA_\PP(*\infty))$.

Recently D. Korotkin and V. Shramchenko \cite{KorotkinS} and Nakayashiki \cite{Nak16} defined the sigma function of every compact Riemann surface as a generalization of Klein's sigma function transcendentally.
Every compact Riemann surface can be characterized by the Weierstrass non-gap sequence, which is described by a numerical semigroup $H$ called \emph{Weierstrass semigroup}.
Nakayashiki defined the sigma function for every compact Riemann surface with Weierstrass semigroup $H$ \cite{Nak16} based on Sato's theory on the universal Grassmannian manifolds (UGM) \cite{SatoN84, SegalWilson}.

In this paper, we use our recent results on the complementary module of the W-curve \cite{KMP2022a} to define the trace operator $p_X$ such that $p_X(P, Q)=\delta_{P,Q}$ for $\varpi_r(P)=\varpi_r(Q)$ for $P, Q \in X\setminus\{\infty\}$.
In terms of them, we express the fundamental 2-form of the second kind $\Omega$ algebraically in Theorem \ref{2th:Sigma}, and finally obtain a connection to Nakayashiki's sigma function by modifying his definition in Theorem \ref{thm:sigma}.
It means an algebraic construction of the sigma function for every W-curve as Baker did for Klein's sigma function following Weierstrass' elliptic function theory.

\bigskip
Contents are as follows:
Section \ref{sec:WCF} reviews the Weierstrass curves (W-curves) based on \cite{KMP2022a}; 1) the numerical semigroup in Subsection \ref{2sc:WSG},  2) Weierstrass canonical form in Subsection \ref{2ssc:WSG}, 3) their relations to the monomial curves in Subsection \ref{2ssc:MCurve}, 4) the properties of the $R_\PP$-module $R_X$ in Subsection \ref{2ssc:R_Pmodule_R_X}, 5) the covering structures in W-curves in Subsection \ref{2sc:Covering}, and 6) especially the complementary module $R_X^\fc$ of $R_X$ in Subsection \ref{sec:CompM}; the explicit description of $R_X^\fc$ is the first main result in \cite{KMP2022a}.
Section \ref{2ssc:W-normAD} provides the first and the second theorems in this paper on the W-normalized Abelian differentials on $X$.
In Subsection \ref{2ssc:W-norm_nuI}, we review the second main result in  \cite{KMP2022a} on the W-normalized Abelian differentials  $\bH^0(X, \cA_X(*\infty))$, which contains the Abelian differentials of the first kind.
Further we extend the trace operator $p\in R_X\otimes_{R_\PP} R_X$ introduced in \cite{KMP2022a} to $R_X\otimes_{\CC}R_X$ in Subsection \ref{2ssec:Ext_p}.
The trace operator $p$ enables us to define a proper one-form $\Sigma$ and its differential $d\Sigma$ in Subsection \ref{2ssc:Wdiff23}.
After investigating $d \Sigma$, we find the W-normalized Abelian differentials of the second kind in Theorems \ref{2th:dSigma} and \ref{2th:dRhamcoh}. 
Further, in Theorem \ref{2th:Sigma}, we mention our results on the fundamental differential of the second kind $\Omega$ in our W-curves and the Abelian differentials of the third kind.
We obtain the generalized Legendre relation in Proposition \ref{2pr:L-rel}.
Using them, we show the connection of the sigma function for $X$ with $R_X^\fc$ in the W-curves $X$ in Section \ref{sc:SigmaFW}.
As studied in \cite{KMP16}, we introduce the shifted Abelian integrals in Subsection \ref{2ssc:ShiftedAI}.
Subsection \ref{ssc:Rtheta} shows the properties of the Riemann theta functions of W-curves and its Riemann-Kempf theorem as in Proposition \ref{3th:RieKempf_theta}.
In Subsection \ref{ssc:SigmaF}, we define the sigma function for a W-curve $X$ by modifying the definition of Nakayashiki \cite[Definition 9]{Nak16}, and show its properties in Theorem \ref{thm:sigma} as our main results in this paper.

\bigskip
\noindent
{\bf{Acknowledgment:}}
The second author is grateful to Yohei Komori for valuable discussions in his seminar and to Takeo Ohsawa and Hajime Kaji for helpful comments. 
He also thanks Takao Kato for informing us of reference \cite{CoppensKato}. 
At the end of May 2022, since this paper was almost completed, the authors decided  that it would be submitted at the beginning of July and would go through a month of checks by typos and others.
In the meantime, Emma Previato, the author of this paper and the special issue editor of this special issue, passed away on 29 June 2022. 
The first two authors sincerely wish that the great mathematician and their friend and collaborator Emma Previato rest in peace.

He was supported by the Grant-in-Aid for Scientific Research (C) of Japan Society for the Promotion of Science Grant, No.21K03289.

\bigskip

\bigskip

\section{Weierstrass canonical form and Weierstrass curves  (W-curves)}
\label{sec:WCF}

\subsection{Numerical and Weierstrass semigroup}\label{2sc:WSG}

This subsection is on numerical and Weierstrass semigroups based on \cite{ADGS2016, KMP13}.
An additive sub-monoid of the monoid of the non-negative integers $\NN_0$ is called \emph{numerical semigroup} if its compliment in $\NN_0$ is a finite set.
In this subsection, we review the numerical semigroups associated with algebraic curves.

In general, a numerical semigroup $H$ has a unique (finite) minimal set of generators,  $M=M(H)$, ($H=\langle M\rangle$) and the finite cardinality $g$  of $H^\fc=\NN_0\setminus H$; $g$ is the genus of $H$ or $H^\fc$ and $H^\fc$ is called gap-sequence.
We let $r_{\mathrm{min}}(H)$ be the smallest positive integer of $M(H)$.
We call the semigroup $H$ \emph{an $r_{\mathrm{min}}(H)$-semigroup}, so that $\langle 3,7,8\rangle$ is a  $3$-semigroup and $\langle 6,13,14,15,16\rangle$ is a $6$-semigroup. 
Let $N(i)$ and $N^\fc(i)$ be the $i$-th ordered element of $H=\{N(i)\ |\ i \in \NN_0\}$ and $H^\fc=\{N^\fc(i)\ |\ i = 0, 1, \ldots, g-1\}$ satisfying $N(i) < N(i+1)$, and $N^\fc(i) < N^\fc(i+1)$ respectively.

Further the conductor $c_H$ of $H$ is defined by the minimal natural number satisfying $c_H + \NN_0 \subset H$.
The number $c_H-1$ is known as the Frobenius number, which is the largest element of $H^\fc$, i.e., $c_H=N^\fc(g-1)+1$.

By letting the row lengths be $\Lambda_i = N^\fc(g-i) - g+i$, ($i=1, 2, \cdots, g$), we have the Young diagram of the semigroup, $\Lambda:=(\Lambda_1, \ldots, \Lambda_g)$, ($\Lambda_i \le \Lambda_{i+1}$).
The Young diagram $\Lambda$ is a partition of $\displaystyle{\sum_i\Lambda_i}$.
We say that such a Young diagram is associated with the numerical semigroup. 
If for a given Young diagram $\Lambda$, we cannot find any numerical semigroups $H$ such that $\Lambda_i = N^\fc(g-i) - g + i$, we say that $\Lambda$ is not associated with the numerical semigroup. 
It is obvious that in general, the Young diagrams are not associated with the numerical semigroups. 
We show the examples of the Young diagrams associated with numerical semigroup:
\begin{gather}
\begin{matrix}
\yng(2,2,1,1),&
\yng(6,3,3,2,1,1,1,1),&
\yng(6,3,3,3,1,1,1,1).\\
 \langle3, 7, 8\rangle &
 \langle5,7,11\rangle
 &\langle5, 6, 14\rangle\\
\end{matrix}
\label{2eq:yng1}
\end{gather}
The Young diagram and the associated numerical semigroup are called symmetric if the Young diagram is invariant under reflection across the main diagonal.
It is known that the numerical semigroup is symmetric if and only if $2g-1$ occurs in the gap sequence.
It means that if $c_H=2g$, $H$ is symmetric. 

We obviously have the following proposition:
\begin{proposition}\label{2pr:N(n)}
The following hold:
\begin{enumerate}
\item $N(n) - n \le g$ for every $n\in \NN_0$,

\item 
$N(n) - n = g$ for $N(n)\ge c_X=N(g)$ or $n\ge g$, 

\item $N(n) - n < g$ for $0\le N(n)< c_X$ or $n<g$,

\item $\#\{n \ |\ N^\fc(n) \ge g\}=\#\{n \ |\ N(n) < g\}$

\item for $N(i) < N^\fc(j)$, $N^\fc(j)-N(i)\in H^\fc$, and

\item when $H$ is symmetric, $c_H=N(g)=2g$ and
$c_H - N(i)-1=N^\fc(g-i-1)$ for $0\le i\le g-1$.

\end{enumerate}
\end{proposition}

\begin{proof} {\it{1}}-{\it{3}} and {\it{5}} are obvious.
Noting $\# H^\fc =g$, {\it{4}} means that what is missing must be filled later for $H^\fc$. {\it{6}} is left to \cite{ADGS2016}.\qed
\end{proof}

The length $r_\Lambda$ of the diagonal of the Young diagram $\Lambda$ is called the rank of $\Lambda$.
The number of boxes below and to the right of the $i$-th box of the diagonal from lower right to upper left are assumed $a_{r_\Lambda-i+1}$ and $b_{r_\Lambda-i+1}$ respectively.
The Young diagram is represented by $(a_{r_\Lambda}, \ldots, a_2, a_1; b_{r_\Lambda}, \ldots,b_2, b_1)$, which is known as Frobenius representation or characteristics of $\Lambda$.
Then $\ell_i := a_i + b_i +1$ is called the hook length of the characteristics.
For $\Lambda = (6,3,3,2,1,1,1,1)$ associated with $H=\langle 5,7,11\rangle$, it is $(0, 2, 7; 0,1,5)$ and rank of $\Lambda$ is three.
For $\Lambda = (6,3,3,3,1,1,1,1)$ associated with $H=\langle 5,6,14\rangle$, it is $(1, 2, 7; 0,1,5)$ and rank of $\Lambda$ is three.
We show their examples of the Young diagram:

\begin{gather}
\begin{matrix}
\young(\ -----,|\ -,||\ ,||,|,|,|,|),
&
\young(\ -----,|\ -,||\ ,|||,|,|,|,|),
\\
(0, 2, 7; 0,1,5),
 &(1, 2, 7; 0,1,5),\\
 \langle5,7,11\rangle
 &\langle5,6,14\rangle.\\
\end{matrix}
\label{2eq:yng3}
\end{gather}

\begin{definition}\label{2df:Lambda^k}
For a given Young diagram $\Lambda=(\Lambda_1, \Lambda_2, \cdots, \Lambda_g)$, we define
$$
\Lambda^{(k)}:=(\Lambda_1, \Lambda_2, \cdots, \Lambda_k), \quad
\Lambda^{[k]}:=(\Lambda_{k+1}, \Lambda_{k+2}, \cdots, \Lambda_g).
$$
\end{definition}

We show the properties of the Young diagram associated with the numerical semigroup $H$ in the following lemma, which is geometrically obvious:
\begin{lemma}\label{2lm:Lambda^k}
\begin{enumerate}

\item If we put the number on the boundary of the Young diagram $\Lambda$ from the lower to upper right as in (\ref{2eq:yng4}), each number in the right side box of the $i$-th row corresponds to the gap number in $N^\fc(g-i)\in H^\fc$ or $\Lambda_i +g - i$, i.e., $N^\fc(g-i)\in H^\fc=\Lambda_i +g - i$, for $0<i\leq g$, and each number in the top numbered box of the $i$-th column corresponds to $N(i)$ for $0<i<\Lambda_1$.

\item The hook length $\ell_i$ is given by $\ell_i = N^\fc(g-i) - N(i-1)$\, $(i=1, \ldots, r_\Lambda)$, which belongs to $H^\fc$, and thus let $L_i \in \{0, 1, \ldots, g-1\}$ such that $N^\fc(g-L_i)=\ell_i$. 
Then $\Lambda_{L_i} + g-L_i=\ell_i$, and $L_1 =1$.

\item For $\Lambda:=(\Lambda_1, \ldots, \Lambda_g)$ associated with the numerical semigroup $H:=\langle M \rangle$, the truncated Young diagram $\Lambda^{[k]}:=(\Lambda_{k+1}, \ldots, \Lambda_g)$ is also associated with a numerical semigroup $H'$, i.e,, $H'$ is generated by 
$$
M \cup \{N^\fc(g-k), N^\fc(g-k+1),  \ldots, N^\fc(g-1)\},
$$
though the generators are not minimal in general.
Further its compliment $H^{\prime \fc}$ is a subset of $H^\fc$, i.e., $H^{\prime \fc}\subset H^\fc$.

\item For $\Lambda:=(a_{r_\Lambda}, \ldots, a_2, a_1; b_{r_\Lambda}, \ldots,b_2, b_1)$ associated with the numerical semigroup $H$, $\Lambda[\ell]:=(a_{r_\Lambda}, \ldots, a_{\ell+1}, a_\ell; b_{r_\Lambda}, \ldots,b_{\ell+1}, b_\ell)$, and then $\Lambda[\ell]$ is, in general, not associated with the numerical semigroup unless $\ell = 1$.

Further by letting $\displaystyle{|\Lambda|:=\sum_{i=1}^{r_\Lambda} (a_i + b_i+1)}$, we have $\displaystyle{|\Lambda|:=\sum_{i=1}^{g_\Lambda} \Lambda_i}$.

\end{enumerate}
\end{lemma}

$\Lambda^{[1]}$ of the numerical semigroup $\langle5,7,11\rangle$ is associated with the numerical semigroup $\langle5,7,11,13\rangle$.
We show their examples of the Young diagram:
\begin{gather}
\begin{matrix}
\young(\ 5,34,2,1),&
\young(\ \ 9,\ 78,56,4,3,2,1),&
\young(\ \ 9,\ \ 8,567,4,3,2,1).\\
 \langle3, 7, 8\rangle &
 \langle5,7,11,13\rangle
 &\langle5,6,13,14\rangle\\
\end{matrix}
\label{2eq:yng4}
\end{gather}

In this paper, we mainly consider the $r$-numerical semigroup, $H$.
We introduce the tools as follows:

\begin{definition}\label{2df:NSG1}
\begin{enumerate}

\item
Let $\ZZ_r:= \{ 0, 1, 2, \ldots, r-1\}$ and $\ZZ_r^\times := \ZZ_r \setminus\{0\}$.

\item
Let $\tfe_i :=\min\{h \in H \ |\ i \equiv h \ \mbox{ mod }\ r\}$,
$i \in \ZZ_r$. 

\item Let $\tfE_H:=\{\tfe_i \ |\ i \in \ZZ_r\}$ be the standard basis of $H$.
Further we define the ordered set $\fE_H:=\{\fe_i \in \tfE_H \ | \ \fe_i < \fe_{i+1}\}$, and $\fE_H^\times := \fE_H\setminus \{0\}$, e.g., $\fe_{0}=\tfe_0=0$.

\item Let $\lH^\fc:= H^\fc \bigcup (-\NN)$, where $\NN:=\NN_0\setminus\{0\}$.

\end{enumerate}

\end{definition}

We have the following elementary but essential results \cite[Lemma 2.8, 2.9]{KMP2022a}:

\begin{lemma}\label{2lm:NSG1}

For $a \in \NN_0$, we define
$$
[a]_r:= \{ a + k r \ | \ k \in \NN_0\}, \quad
\overline{[a]}_r^\fc:=
\{ a - k r \ | \ k  \in \NN\}, \quad
[a]_r^\fc:=\overline{[a]}_r^\fc\cap \NN.
$$

\begin{enumerate}
\item We have the following decomposition;

\begin{enumerate}

\item 
$\displaystyle{
    H = \bigoplus_{i\in \ZZ_r}  [\fe_i]_r,
}$ 

\item 
$\displaystyle{
    \lH^\fc  = \bigoplus_{i\in \ZZ_r}  \overline{[\fe_i]}_r^\fc
}$,\  $H\bigcup \lH^\fc=\ZZ$,

\item 
$\displaystyle{
    H^\fc  = \bigoplus_{i\in \ZZ_r}  [\fe_i]_r^\fc
= \bigoplus_{i\in \ZZ_r^\times}  [\fe_i]_r^\fc
}$,\ $H\bigcup H^\fc=\NN_0$,
\end{enumerate}

\item for every $x_i \in [\fe_i]_r$ $(i\in \ZZ_r)$,
$$
\{x_i \mbox{ modulo } r\ |\ i \in \ZZ_r\}=\ZZ/r \ZZ,
$$
and, especially for $x \in [\tfe_i]_r$, $x = i$ modulo $r$.
\end{enumerate}
\end{lemma}

The following is obvious:

\begin{lemma}\label{2lm:rs_iris}
For the generators $r$ and $s$ in the numerical semigroup $H$, there are positive integers $i_s$ and $i_r$ such that
$
   i_s s- i_r r = 1.
$
\end{lemma}

\medskip

\subsection{Weierstrass canonical form}\label{2ssc:WSG}

We recall the ``Weierstrass canonical form'' (``Weierstrass normal form'') based on \cite{KMP16, KMP19, KM2020, KMP2022a}, which is a generalization of Weierstrass' standard form for elliptic curves and whose origin came from Abel's insight; Weierstrass investigated its primitive property \cite{Wei67, WeiWIV}.
Baker \cite[Ch. V, \S\S 60-79]{Baker97} gives its complete review, proof, and examples, but we refer to Kato \cite{Kato80, CoppensKato}, who also produces this representation from a modern viewpoint.

\begin{proposition} {\rm{\cite{Kato80, CoppensKato}}} \label{2pr:WCF}
For a pointed curve $(X,\infty )$ with Weierstrass semigroup $H_X:=H(X,\infty)$
for which  $r_{\mathrm{min}}(H_X)=r$, and $\fe_i\in \fE_{H_X}$, $(i \in \ZZ_r^\times)$ in Definition \ref{2df:NSG1}, and we let 
$\displaystyle{s:=\min_{i\in \ZZ_r^\times}\{\fe_i \in \fE_{H_X}}$ ${ |\ (\fe_i, r)=1\}}$ and $s=\fe_{\ell_s}$. $(X,\infty)$ is defined by an irreducible equation,
\begin{equation}
f_X(x,y)= 0,
\label{2eq:WCF1a}
\end{equation}
for a polynomial $f_X\in \CC[x,y]$ of type,
\begin{equation}
f_X(x,y):=y^r + A_{1}(x) y^{r-1}
+ A_{2}(x) y^{r-2}
+\cdots + A_{r-1}(x) y
 + A_{r}(x),
\label{2eq:WCF1b}
\end{equation}
where the $A_i(x)$'s are polynomials in $x$, $A_0=1$,
$\displaystyle{
A_i = \sum_{j=0}^{\lfloor i s/r\rfloor} \lambda_{i, j} x^j}$,
and $\lambda_{i,j} \in \CC$, $\lambda_{r,s}=-1$.
\end{proposition}

In this paper we call the curve in Proposition \ref{2pr:WCF} a {\emph{Weierstrass curve}} or a {\emph{W-curve}}. 
The Weierstrass canonical form characterizes the W-curve, which has only one infinity point $\infty$. 
The infinity point $\infty$ is a Weierstrass point if $H_X^\fc=H^\fc(X,\infty)=\{N(i)\}$ differs from  $\{1, 2, \cdots, g\}$ \cite{FarkasKra}.
Since every compact Riemann surface of the genus, $g(>1)$, has a Weierstrass point whose Weierstrass gap sequence with genus $g$ \cite{ACGH85}, it characterizes the behavior of the meromorphic functions at the point, and thus there is a Weierstrass curve which is bi-rationally equivalent to the compact Riemann surface.

Further Proposition \ref{2pr:WCF} is also applicable to a pointed compact Riemann surface $(Y, P)$ of genus $g$ whose point $P$ is non-Weierstrass point rather than the Weierstrass point; its Weierstrass gap sequence at $P$ is $H^\fc(Y, P)=\{1, 2, \ldots, g\}$.
Even for the case, we find the Weierstrass canonical form $f_X$ and the W-curve $X$ with $H_X^\fc=\{1, 2, \ldots, g\}$ which is bi-rational to $Y$.

\begin{remark}\label{2rm:g_Xc_X}
{\rm{
Let $R_{X^\circ}^\circ := \CC[x,y]/(f_X(x,y))$ for (\ref{2eq:WCF1a}) and its normalized ring be $R_X^\circ$ if $X^\circ:=\Spec R_{X^\circ}^\circ$ is singular.
$R_X^\circ$ is the coordinate ring of the affine part of $X\setminus\{\infty\}$ and we identify $R_X^\circ$ with $R_X=\bH^0(X,\cO_X(*\infty))$.
Then the quotient field $\CC(X):=\cQ(R_X)$ of $R_X$ is considered as an algebraic function field on $X$ over $\CC$.

By introducing $R_\PP:=\bH^0(\PP, \cO_\PP(*\infty))=\CC[x]$ and its quotient field $\CC(x):=\cQ(R_\PP)$, $\cQ(R_X)$ is considered a finite extension of $\cQ(R_\PP)$. 
We regard $R_X$ as a finite extended ring of $R_\PP$ of rank $r$, e.g., $R_{X^\circ}^\circ=R_\PP[y]/(f_X(x,y))$ as mentioned in Subsection \ref{2sc:Covering} \cite{Kunz2004}.

For the local ring $R_{X,P}$ of $R_X$ at $P \in X$, we have the ring homomorphism,
$\varphi_{P}: R_X\to R_{X,P}$.
We note that  $R_{X, \infty}$ plays crucial roles in the Weierstrass canonical form.
We let the minimal generator $M_X =\{ r_1, r_2, \ldots, r_{m_X}\}$ of the numerical semigroup $H_X=H(X, \infty)$.
The Weierstrass curve admits a local cyclic $\fC_r=\ZZ/r\ZZ$-action at $\infty$ c.f., Subsection \ref{2ssc:MCurve}.
The genus of $X$ is denoted by $g_X$, briefly $g$ and the conductor of $H_X$ is denoted by $c_X:=c_{H_X}$; the Frobenius number $c_X-1$ is the maximal gap in $H(X, \infty)$.
We let, $\lH_X^\fc:=\ZZ\setminus H_X$.
}}
\end{remark}

\subsubsection{Projection from $X$ to $\PP$}\label{2sssc:X_to_PP}

There is the  natural projection, 
\begin{equation}
\varpi_{r}:X \to \PP, \qquad(
\varpi_{r}(x,y_{r_2}, \ldots, y_{r_{m_X}}) = x=y_r)
\label{2eq:varphi_ri}
\end{equation}
such that $\varpi_{r}(\infty) = \infty\in \PP$.

Let $\{y_{\bullet}\}:=\{y_{s}=y_{r_2}, y_{r_3}, \ldots, y_{r_{m_X}}\}$ and $\CC[x, y_{\bullet}]:=\CC[x, y_{s}=y_{r_2}, y_{r_3}, \ldots, y_{r_{m_X}}]$.

\subsection{The monomial curves and W-curves}
\label{2ssc:MCurve}

This subsection shows the monomial curves and their relation to W-curves based on \cite{KMP16, KMP19, KM2020}.

For a given W-curve $X$ with the Weierstrass semigroup $H=H_X$, and its generator $M_X=\{r=r_1, r_2, \ldots, r_{m_X}\}$, the behavior of singularities of the elements in $R_X$ at $\infty$ is described by a monomial curve $X_H^Z$.
For the numerical semigroup $H=\langle M_X\rangle$, the \emph{numerical semigroup ring} $R_H$ is defined as $R_H:=\CC[ z^{r_1}, z^{r_2}, \cdots, z^{r_{m_X}}]$.

Following a result of Herzog's \cite{Herzog70}, we recall the well-known proposition for a polynomial ring $\CC[Z]:=\CC[Z_{r_1}, Z_{r_2}, \cdots, Z_{r_{m_X}}]$.
\begin{proposition} \label{2pr:R_H}
For the $\CC$-algebra homomorphism $ \tvarphi^Z_H : \CC[Z]  \to R_H$, the kernel of $\tvarphi^Z_H$ is generated by $f^H_i = 0$ $(i=1,2,\ldots,k_X)$ of a certain binomial $f^H_i \in \CC[Z]$ and a positive integer $k_X$ $\displaystyle{\left(m_X-1\le k_X< \infty\right)}$, i.e., $\ker\tvarphi^Z_H=(f^H_1, f^H_2, \ldots, f^H_{k_X})$, and
$$
R_H \simeq \CC[Z]/ \ker\tvarphi^Z_H=:R^Z_H.
$$
\end{proposition}

We call $R^Z_H=\CC[Z]/ \ker\tvarphi^Z_H$ a \emph{monomial ring}.
Sending $Z_r$ to $1/x$ and $Z_{r_i}$ to $1/y_{r_i}$, the monomial ring $R^Z_H$
determines the structure of gap sequence of $X$ at $\infty$ \cite{Herzog70, Pinkham74}.
Bresinsky showed that $k_X$ can be any finitely large number if $m_X>3$ \cite{Bresinsky}.

Let $X_H:=\Spec R_H^Z$,which we call a \emph{monomial curve}.
We also define the ring isomorphism on $R_H^Z$ induced from $\tvarphi^Z_H$, which is denoted by $\varphi^Z_H$,
\begin{equation}
\varphi^Z_H: \CC[Z]/ \ker\tvarphi^Z_H=R^Z_H \to R_H.
\label{2eq:varphi^Z_H2}
\end{equation}

Further, we let $\{Z_{\bullet}\}:=\{Z_{r_2}, Z_{r_3}, \ldots, Z_{r_{m_X}}\}$, and $\CC[Z_r, Z_{\bullet}]:=\CC[Z_r, Z_{r_2}, Z_{r_3}, \ldots, Z_{r_{m_X}}]$.
A monomial curve is an irreducible affine curve with $\GG_\fm$-action, where $\GG_\fm$ is the multiplicative group of the complex numbers; $Z_a \mapsto g^{a} Z_a$ for $g\in \GG_\fm$, and it induces the action on the monomial ring $R_H^Z$.

The following cyclic action of order $r$ plays a crucial role in this paper.

\begin{lemma} \label{2lm:fC_r_action}
The cyclic group $\fC_r$ of order $r$ acts on the monomial ring $R_H^Z$;
the action of the generator $\hzeta_r \in \fC_r$ on $Z_a$ is defined by sending $Z_a$ to $\zeta_r^{a}Z_a$, where $\zeta_r$  is a primitive $r$-th root of unity.
By letting $\fr_i^*:=(r, r_i)$, $\fr_i := r/\fr_i*$, and $\Bfr_i:=r_i/\fr_i^*$, the orbit of $Z_{r_i}$ forms $\fC_{\fr_i}$; especially for the case that $(r, r_i) = 1$, it recovers $\fC_r$.

Thus in $R_H^Z$, 
\begin{equation}
f_H^{(j)}(Z_r, Z_{r_j})= 0, \quad f_H^{(j)}:=Z_{r_j}^{\fr_j}-Z_r^{\Bfr_j},
\quad (j = 2, \ldots, m_X).
\label{2eq:f_H^j}
\end{equation}
\end{lemma}

For a ring $R$, let its quotient field be denoted by $\cQ(R)$.
Further we obviously have the identity in $\cQ(R_H^Z \otimes_{\CC[Z_r]} R_H^Z)$,
\begin{equation}
   p_H^{(j)}(Z_{r_j}, Z_{r_j}'):=\frac{1}{\fr_j Z_{r_j}^{\fr_j-1}}
          \frac{Z_{r_j}^{\fr_j}-Z_{r_j}^{\prime\, \fr_j}}
               {Z_{r_j}-Z_{r_j}^{\prime}}=\left\{
\begin{array}{rcl}
1 & \mbox{  for  } &Z_{r_j}= Z_{r_j}',\\
0 & \mbox{  for  } & Z_{r_j}\neq Z_{r_j}',
\end{array}\right.
\label{2eq:p_H^j}
\end{equation}
as a $\CC[Z_r]$-module.

\bigskip
Corresponding to the standard basis of $H_X$ in Definition \ref{2df:NSG1}, we find the monic monomial $\fZ_{\fe_i} \in \CC[Z]$ such that $\varphi^Z_H(\fZ_{\fe_i}) = z^{\fe_i}$, and the standard basis $\{\fZ_{\fe_i}\ | i \in \ZZ_r\}$; $\fZ_{\fe_0}=1$. 

\begin{lemma}\label{2lm:Z_standardbasis}
The $\CC[Z_r]$-module $\CC[Z]$ is given by
$$
\CC[Z]=\CC[Z_r]\oplus \CC[Z_r]\fZ_{\fe_1}\oplus \cdots
\oplus \CC[Z_r]\fZ_{\fe_{r-1}},
$$
and thus $\fZ_H:=\{\fZ_{\fe_0}, \fZ_{\fe_1}, \ldots, \fZ_{\fe_{r-1}}\}$ is the basis of the $\CC[Z_r]$-module $\CC[Z]$.
Then there is a monomial $b_{i j k}\in \CC[Z_r]$ such that
$$
\fZ_{\fe_i}\fZ_{\fe_j}
=\sum_{k\in \ZZ_r}b_{i j k}\fZ_{\fe_k}.
$$
\end{lemma}

\begin{lemma}\label{2lm:p_HZ_r}
By defining an element in $\cQ(R_H^Z \otimes_{\CC[Z_r]} R_H^Z)$ by
$
p_H(Z_r, Z_\bullet, Z_\bullet'):=
\displaystyle{\prod_{j=2}^{m_X}}
p_H^{(j)}(Z_{r_j}, Z_{r_j}'),
$
we have the identity
\begin{equation}
   p_H(Z_r, Z_\bullet, Z_\bullet')=\left\{
\begin{array}{rcl}
1 & \mbox{  for  } &Z_\bullet= Z_\bullet',\\
0 & \mbox{  for  } & Z_\bullet\neq Z_\bullet',
\end{array}\right.
\label{2eq:tp_H}
\end{equation}
as $\CC[Z_r]$-module properties and 
\begin{equation}
\varphi_Z^H(\tp_H(Z_r, Z_\bullet, Z_\bullet'))
= \frac{1}{r}\sum_{i=0}^r \frac{z^{\prime\,\fe_i}}{z^{\fe_i}}.
\label{2eq:varphi_tp_H}
\end{equation}
\end{lemma}

\bigskip

To construct our curve $X$ from $R_H$ or $\Spec R_H$, we could follow Pinkham's strategy \cite{Pinkham74} with an irreducible curve singularity with $\GG_\fm$ action, though we will not mention it in this paper.
Pinkham's investigations provide the following proposition \cite{Pinkham74}\cite[Proposition 3.7]{KMP2022a}:
\begin{proposition} \label{2pr:R_H^Z}
For a given W-curve $X$ and its associated monomial ring,
$R_H^Z = \CC[Z]/ (f^H_1, f^H_2, \ldots,$ $ f^H_{k_X})$, there are a surjective ring-homomorphism \cite[p.80]{ADGS2016}
$$
\varphi^X_H: R_X \to R_H^Z
$$
such that $R_X/\fm_{\bA}$ is isomorphic to $R_H^Z$, where $\fm_{\bA}$ is the maximal ideal $(\lambda_{ij})$ in coefficient ring $\CC[\lambda_{ij}]$, and $\varphi^X_H(y_{r_i})=Z_{r_i}$, and a polynomial $\{f^{X}_i\}_{i = 1, \ldots, k_X}\in \CC[x,y_\bullet]$ satisfying
\begin{enumerate}

\item $\varphi^X_H(f^X_i)(=\varphi^X_H(f^X_i$ modulo $\fm_{\bA}))=f^H_i$ for $(i=1,\ldots,k_X)$, 

\item 
the affine part of $R_X$ is given by $R_X=\CC[x,y_{\bullet}]/(f^X_1, f^X_2, \ldots, f^X_{k_X})$,  and 

\item the rank of the matrix
$\displaystyle{
\left(
\frac{\partial f^X_i}{\partial y_{r_j}}\right)_{
i=1, 2, \ldots, k_X, j = 1, 2, \ldots, m_X}}$ is $m_X-1$ for every point $P$ in $X$.

\end{enumerate}
\end{proposition}

\begin{definition} \label{2df:wt_RX}
\begin{enumerate}

\item
Recalling Lemma \ref{2lm:rs_iris}, we define {\emph{arithmetic local parameter}}  at $\infty$ by
$\displaystyle{
t = \frac{x^{i_r}}{y^{i_s}}
}$ \cite{Onishi18}.

\item
The degree at $\cQ(R_{X,\infty})$ as the order of the singularity with respect to $t$ is naturally defined by
\begin{equation}
\wt=\deg_\infty: \cQ(R_X) \to \ZZ,
\label{2eq:wt_RX}
\end{equation}
which is called \emph{Sato-Weierstrass weight} \cite{Wei67}.

\item In the ring of the formal power series $\CC[[t_1, \ldots, t_\ell]]$, 
we define the symbol $d_{>n}(t_1, \ldots, t_\ell)$ so that it belongs to the ideal,
$$
d_{>n}(t_1, \ldots, t_\ell)\in \{\sum a_{i_1,\ldots, i_\ell} t_1^{i_1} 
\cdots t_\ell^{i_\ell}\ |\ a_{i_1,\ldots, i_\ell}=0 
\mbox{ for } i_1+\cdots+i_\ell \leq n \}.
$$
\end{enumerate}

\end{definition}

The weight of $y_{r_i}$ is given by
$$
\wt(y_{r_i}) = -r_i, \quad (i=1,2,\ldots,{m_X}),
\qquad
y_{r_i}=\frac{1}{t^{r_i}}(1+d_{>0}(t)).
$$

\begin{lemma}\label{2lm:RX_Cphi}
We have the decomposition of $R_X$ as a $\CC$-vector space,
\begin{equation}
R_X = \bigoplus_{i=0} \CC \phi_i,
\label{2eq:RXphi}
\end{equation}
where $\phi_i$ is a monomial in $R_X$ satisfying the inequalities $-\wt\,\phi_i < -\wt\,\phi_j$ for $i<j$, i.e., $\phi_0 = 1$, $\phi_1 = x$, $\ldots$.
\end{lemma}

Further by assigning a certain weight on each coefficient $\lambda_{i,j}$ in (\ref{2eq:WCF1a}) so that (\ref{2eq:WCF1a}) is a homogeneous equation of weight $rs$, we also define another weight, 
\begin{equation}
\wt_\lambda: R_X \to \ZZ.
\label{2eq:wt_lambda}
\end{equation}

\begin{definition}\label{2df:t_SR_N}
We define $S_X:=\{\phi_i\ |\ i = 0,1,2 \ldots\}$ by the basis of $R_X$ as in (\ref{2eq:RXphi}).
\end{definition}

Then $N(i)=-\wt(\phi_i)$, for $\{N(i)\ |\ i \in \NN_0\}=H_X$.

\begin{lemma}\label{2lm:varphiinf}
Let $t$ be the arithmetic local parameter at $\infty$ of $R_X$.
\begin{enumerate}

\item By the isomorphism $\displaystyle{\varphi_{\mathrm{inv}}:z \mapsto \frac{1}{t}}$, $\varphi_{\mathrm{inv}}(R_H)(\cong R_H^Z)$ is a subring of $\displaystyle{\CC\left[\frac{1}{t}\right]}$; for $g(z)\in R_H$, $\displaystyle{g\left(\frac{1}{t}\right)\in \CC\left[\frac{1}{t}\right]}$.

\item There is a surjection of ring $\varphi_\infty: R_{X}\to R_H(\cong R_H^Z)$; for $f\in R_X$, there is $g(z)\in R_H$ such that
$$
       (f)_\infty = g
\left(\frac{1}{t}\right)(1+d_{>0}(t)) \in R_{X,\infty},
$$
where $(f)_\infty$ means the germ at $\infty$ or $(f)_\infty \in R_{X, \infty}$ via $\varphi_H^X$ in Proposition \ref{2pr:R_H^Z}.
It induces the surjection $R_{X,\infty}\to R_H(\cong R_H^Z)$.

\end{enumerate}
\end{lemma}

\begin{proof}
By letting $g=\varphi_H^Z \circ\varphi_H^X(f)$, the existence of $g$ is obvious.\qed
\end{proof}

\subsection{$R_\PP$-module $R_X$}
\label{2ssc:R_Pmodule_R_X}

$R_X$ is an $R_\PP$-module, and its affine part is given by the quotient ring of $R_\PP[y_{\bullet}]$.

We recall Definition \ref{2df:NSG1} and Lemma \ref{2lm:Z_standardbasis}, and apply them to W-curves and then we obtain the following \cite[Proposition 3.11]{KMP2022a}.

\begin{proposition}\label{2pr:RP-moduleRX}
For $\fe_i \in \fE_{H_X}$, we let $\fy_{\fe_i}$ be the monic monomial in $R_X$ whose weight is $-\fe_i$, $(\fy_{\fe_0}=1)$ satisfying
$$
R_X = R_\PP \oplus \bigoplus_{i=1}^{r-1} R_\PP \fy_{\fe_i}
=\bigoplus_{i=0}^{r-1} R_\PP \fy_{\fe_i}
=\langle \fy_{\fe_0}, \fy_{\fe_1}, \ldots, \fy_{\fe_{r-1}}\rangle_{R_\PP}
$$
with the relations,
\begin{equation}
\fy_{\fe_i} \fy_{\fe_j} =\sum_{k=0}^{r-1} \fa_{i j k} \fy_{\fe_k},
\label{2eq:fyfy=fy}
\end{equation}
where $\fa_{i j k} \in R_\PP$, $\fa_{i j k}=\fa_{i j k}$, especially $\fa_{0j k}=\fa_{j0k}=\delta_{j k}$.
\end{proposition}

\subsection{The covering structures in W-curves}
\label{2sc:Covering}

We will follow \cite{Kunz2004, Stichtenoth} to investigate the covering structure in W-curves.

\subsubsection{Galois covering}
\label{2sc:GaloisCover}

As mentioned in Remark \ref{2rm:g_Xc_X}, let us consider the Riemann sphere $\PP$ and $R_\PP=\bH^0(\PP,$ $ \cO_{\PP}(*\infty))$. We identify $R_\PP$ with its affine part $R_\PP^\circ=\CC[x]$ and its quotient field is denoted by $\CC(x)=\cQ(R_\PP)$.
The quotient field $\cQ(R_X)=\CC(X)$ of $R_X$ is an extension of the field $\CC(x)$.

Following the above description, we consider the W-curve $X$.
The covering $\varpi_r: X \to \PP$ $((x,y_\bullet) \mapsto x)$ is obviously a holomorphic $r$-sheeted covering. When we have the Galois group on $X$, i.e., $\Gal(\cQ(R_X)/\cQ(R_\PP))=\Aut(X/\PP)=\Aut(\varpi_r)$, it is denoted by $G_X$.
The $\varpi_{r}$ is a finite branched covering.
A ramification point of $\varpi_{r}$ is defined as a point of such that is not biholomorphic at the point.
The image $\varpi_{r}$ of the ramification point is called the branch point of $\varpi_{r}$. The number of the finite ramification points is denoted by $\ell_{\fB}$.

We basically focus on the holomorphic $r$-sheeted covering $\varpi_{x}=\varpi_{r}: X \to \PP$. $G_x$ denotes the finite group action on $\varpi_{r}^{-1}(x)$ for $x \in \mathbb{P}$, refered to as group action at $x$ in this paper:
\begin{definition}\label{2df:fBX}
Let $\fB_X:=\fB_{X,r}$ $=\{B_{i}\}_{i=0, \cdots, \ell_{\fB}}$ and $\fB_{\PP}:=\varpi_r(\fB_X)=\{b_{i}\}_{i=0, \cdots, \ell_{\fB}}$, where $\ell_{\fB}:= \# \fB_X-1$, $B_{0}=\infty\in X$ and $b_i :=\varpi_r(B_i)$.
\end{definition}

\subsubsection{Riemann-Hurwitz theorem}

Let us consider the behaviors of the covering $\varpi_r: X \to \PP$, including the ramification points.
The Riemann-Hurwitz theorem \cite{Kunz2004},
\begin{equation}
2g-2 =-2r+ \sum_{i=1}^{\ell_{\fB}} (e_{B_i}-1)+(r-1)=
\sum_{i=1}^{\ell_{\fB}} (e_{B_i}-1)-(r+1),
\label{2eq:RHt_wr}
\end{equation}
where $e_{B_i}$ is the ramification index at $B_i$, shows the following:
\begin{corollary}\label{2cr:(dx)} The divisor of $d x$ is given by,
$$
\Div(d x) =\sum_{i=1}^{\ell_{\fB}} (e_{B_i}-1) B_i - (r+1)\infty.
$$
\end{corollary}

\subsubsection{Embedding of $X$ into $\PP^{2(m_X-1)}$}

By identifying $R_X=R_X^\circ=\CC[x, y_\bullet]/(f_1^X, \ldots, f_{k_X}^X)$, it is obvious that $R_{X^\circ}^\circ=$\break $\CC[x, y_{s}]/f_X(x, y_s)$ is a subring of $R_X$ because $R_X$ is a normalized ring of $R_{X^\circ}^\circ$. There is a projection $\varpi_{r,r_2}:X \to X^\circ$. Thus we can find the subring $R_{X^{(i)}}=\CC[x,y_{r_i}]/(f_X^{(i)}(x, y_{r_i})$ of $R_X$ as its normalized ring, as we prove this fact in \cite[Proposition 3.14]{KMP2022a} precisely.
The image of $f_X^{(i)}(x, y_{r_i})$ by $\varphi_H^X$ is (\ref{2eq:f_H^j}).
There are injective ring-homomorphisms $R_\PP\xrightarrow{\iota_r^{(i)}} R_{X^{(i)}}\xrightarrow{\iota_{r,r_i}} R_X$ and they induce the projections $\varpi_{r_i, r} : X \to X^{(i)}$ $((x, y_\bullet) \mapsto (x, y_{r_i}))$ and $\varpi_{r}^{(i)} : X^{(i)} \to \PP$ $((x, y_{r_i}) \mapsto x)$ where $X^{(i)}:=\Spec R_{X^{(i)}}$; They satisfy the commutative diagrams,
\begin{equation}
\xymatrix{ 
 R_X& R_{X^{(j)}}\ar[l]_{\iota_{r_j,r}}\\
R_{X^{(i)}} \ar[u]^{\iota_{r_i,r}}&
R_\PP,\ar[l]^{\iota_{r}^{(i)}}
\ar[lu]^{\iota_r}
\ar[u]_{\iota_{r}^{(j)}}
}\qquad
\xymatrix{ 
 X \ar[dr]^{\varpi_r}\ar[r]^-{\varpi_{r,r_j}}
\ar[d]_{\varpi_{r,r_i}}& X^{(j)} \ar[d]^-{\varpi_{r}^{(j)}} \\
X^{(i)} \ar[r]_{\varpi_{r}^{(i)}} & \PP.
}
\label{2eq:varpi_rr_j}
\end{equation}
Further we also define the tensor product of these rings $R_{X^{(2)}}\otimes_{R_\PP} R_{X^{(3)}}\otimes_{R_\PP} \cdots \otimes_{R_\PP}R_{X^{(m_X)}}$, and its geometrical picture $X^{[m_X-1]}_\PP:=X^{(2)}\times_\PP X^{(3)} \times_\PP\cdots \times_\PP X^{(m_X)}$.
By identifying  $\CC[x,y_{\bullet}]$ $/(f_X^{(2)},$ $\ldots, f_X^{(m_X)})=R_X^{\otimes[m_X-1]}$ with a ring $R_{X^{(2)}}\otimes_{R_\PP} R_{X^{(3)}}\otimes_{R_\PP} \cdots \otimes_{R_\PP}R_{X^{(m_X)}}$, we have the natural projection $\varphi_{R_X^{\otimes[m_X-1]}}:R_X^{\otimes[m_X-1]}\to R_X$, i.e., 
$R_X=R_X^{\otimes[m_X-1]}/(f_1^{X}, \ldots, f_{k_X}^X)$, and the injection $\iota_{R_X^{\otimes[m_X-1]}}:R_\PP\hookrightarrow R_X^{\otimes[m_X-1]}$.
It induces the injection $\iota_{ X^{[m_X-1]}_\PP}: X \to X^{[m_X-1]}_\PP$ and 
the projection $\prod_i \varpi_{X^{(i)}}: X^{[m_X-1]}_\PP \to \PP$.

Moreover we also define the direct product of these rings $R_X^{[m_X-1]}:=R_{X^{(2)}}\times R_{X^{(3)}}\times \cdots \times R_{X^{(m_X)}}$, and its geometrical picture $X^{[m_X-1]}:=\prod X^{(i)}\subset \PP^{2(m_X-1)}$.
Then we have an embedding,
$$
\varphi_{R_X^{[m_X-1]}}:
R_X^{[m_X-1]}
\to R_X, \quad \iota_{ X^{[m_X-1]}}: X \hookrightarrow X^{[m_X-1]} 
(\subset \PP^{2(m_X-1)}).
$$

\subsection{Complementary module $R_X^\fc$ of $R_X$}\label{sec:CompM}

By introducing 
$\displaystyle{
p_X^{(i)}:=\frac{1}{f_{X,y}^{(i)}}\frac{f_{X}^{(i)}(x, y_{r_i})
-f_{X}^{(i)}(x, y_{r_i}')}{y_{r_i}-y_{r_i}'}
}$,
where $\displaystyle{
f_{X,y}^{(i)}(x, y)=\frac{\partial f_{X}^{(i)}(x,y)}{\partial y}}$, as an $R_X$-analog of (\ref{2eq:p_H^j}),
let us consider $\displaystyle{p_{R_X}:=\prod_{j=2}^{m_X} p_X^{(j)}}$ as an element of $\cQ(R_X\otimes_{R_\PP} R_X)$. 
The following is obvious:

\begin{proposition}\label{2pr:p_varpi0}
For $(P,Q) \in X \times_\PP X$, 
$$
p_{R_X}(P,Q)
=\left\{
\begin{array}{rcl}
1 & \mbox{  for  } &P = Q,\\
0 & \mbox{  for  } & P \neq Q,
\end{array}\right. 
$$
\end{proposition}

However, some parts in its numerator and denominator are canceled because they belong to $R_\PP$.
Thus we introduce an element $h(x, y_\bullet, y_\bullet') \in R_X\otimes_{R_\PP} R_X$ such that $h(x, y_\bullet, y_\bullet') /$ $h(x, y_\bullet, y_\bullet)$
reproduces $p_{R_X}$.

\begin{lemma}{\rm{\cite[Lemma 4.14]{KMP2022a}}}\label{2lm:h_RXe}
For a point $(P = (x, y_\bullet), P'=(x, y_\bullet'))\in X \times_\PP X$, there is a polynomial $\tlh_{R_X}(x, y_{\bullet}, y_{\bullet}')\in R_X \otimes_{R_\PP} R_X$ such that
\begin{enumerate}
\item by regarding the element $a$ in $R_X$ as 
$a\otimes 1$ in $R_X \otimes_{R_\PP} R_X$,
$\tlh_{R_X}(x, y_{\bullet}, y_{\bullet})$ and 
$\tlh_{R_X}(x, y_{\bullet}, y_{\bullet}')$ are coprime 
as  elements in $R_X \otimes_{R_\PP} R_X$,

\item for a group action 
$\hzeta\in G_x$,
$\displaystyle{
 \frac{\tlh_{R_X}(x, \hzeta y_{\bullet}, \hzeta y_{\bullet}')}
{\tlh_{R_X}(x, \hzeta y_{\bullet}, \hzeta y_{\bullet})}=
 \frac{\tlh_{R_X}(x, y_{\bullet}, y_{\bullet}')}
{\tlh_{R_X}(x, y_{\bullet}, y_{\bullet})}}$,

\item  it satisfies
$\displaystyle{
\frac{\tlh_{R_X}(x, y_{\bullet}, y_{\bullet}')}
{\tlh_{R_X}(x, y_{\bullet}, y_{\bullet})}=
p_{R_X}(x, y_{\bullet}, y_{\bullet}')
}$, and 

\item
$\displaystyle{
\varphi_H^X(\tlh_{R_X}(x, y_{\bullet}, y_{\bullet}'))
=:\tlh_H(Z_r, Z_\bullet, Z'_\bullet)
}$ ($\wt(\tlh_{R_X}(x, y_{\bullet}, y_{\bullet}'))=d_h$), then 
$$
\frac{\tlh_H(Z_r, Z_\bullet, Z'_\bullet)}{\tlh_H(Z_r, Z_\bullet, Z_\bullet)}
=\tp_H^{(j)}(Z_r, Z_\bullet, Z'_\bullet),
$$
and 
$$
\varphi_H^Z(\tlh_H(Z_r, Z_\bullet, Z'_\bullet))
=z^{d_h}\sum_{i\in \ZZ_r} z^{\prime\, \fe_i} z^{-\fe_i}, \quad
\varphi_H^Z(\tlh_H(Z_r, Z_\bullet, Z_\bullet))=rz^{d_h}.
$$
\end{enumerate}
\end{lemma}

\begin{definition}\label{2df:h_X00}
Let $
h_{X}(x, y_{\bullet}):=\tlh_{R_X}(x, y_{\bullet}, y_{\bullet})$.
\end{definition}

We have the expression of $\tlh_{R_X}(x, y_{\bullet}, y_{\bullet}')$ following
\cite[Proposition 4.16,Lemma 4.20]{KMP2022a}.

\begin{proposition}\label{2pr:h_RXe_ys}
$\tlh_{R_X}\in R_X\otimes_{R_\PP} R_X$ is expressed by
\begin{gather*}
\begin{split}
\tlh_{R_X}(x, y_{\bullet}, y_{\bullet}')
&= \hUpsilon_{0}\cdot 1
+ \hUpsilon_{1}\fy'_{\fe_1}+\cdots+ \hUpsilon_{r-1}\fy'_{\fe_{r-1}}\\
&= 1\cdot\hUpsilon_{0}'
+ \fy_{\fe_1}\hUpsilon_{1}'+\cdots+ \fy_{\fe_{r-1}}\hUpsilon_{r-1}'\\
&= \rfy_{\fe_0}\fy'_{\fe_0} 
+ \rfy_{\fe_1}\fy'_{\fe_1}+\cdots+ \rfy_{\fe_{r-1}}\fy'_{\fe_{r-1}}\\
&\qquad\qquad  +
\mbox{ lower weight terms with respect to $-\wt$}
\end{split}
\end{gather*}
as an $R_\PP$-module.
Here $\fy'_{\fe_0}=1$, and 
each $\hUpsilon_{i}$ has the following properties
\begin{enumerate}
\item 
$\displaystyle{
\hUpsilon_{i}=\sum_{j=0}^{r-1} \fb_{i,j} \fy_{\fe_j},
}$ with certain $\fb_{i,j}\in \CC[x]$, 

\item 
$\displaystyle{
\hUpsilon_{i}= \rfy_{\fe_i}+
\mbox{ lower weight terms with respect to $-\wt$}}$,
where 
$\displaystyle{
\rfy_{\fe_i}
=\tdelta_i(x) \fy_{\fe_{\ell,i}^*}}$
with 
an element $\ell \in \ZZ_r$, and 
a monic polynomial $\tdelta_i(x) \in \CC[x]$ whose weight is 
$-\delta_i r$, (especially, $\rfy_{\fe_0}
=\tdelta_0(x) \fy_{\fe_{\ell,0}^*}=\tdelta_0(x) \fy_{\fe_{\ell}}$)
such that
\begin{eqnarray}
\rfy_{\fe_0}&=&\rfy_{\fe_i}\fy_{\fe_i}+
\mbox{ lower weight terms with respect to $-\wt$}\nonumber
\label{2eq:fy*0_fy*i}
\end{eqnarray} for 
$i \in \ZZ_r$, $\wt(\rfy_{\fe_i})=-(d_h - \fe_i)$ where $b_{i, j}$ is a certain element in 
$R_\PP$ for $(i,j)$, and 

\item $
\langle \hUpsilon_{0}, \hUpsilon_{1}, \ldots, \hUpsilon_{r-1}\rangle_{R_\PP}
=
\langle \hUpsilon_{1}, \ldots, \hUpsilon_{r-1}\rangle_{R_X}.
$

\end{enumerate}
\end{proposition}

We introduce more convenient quantities $\hfy_{\fe_i}$ $(i=0, 1, \ldots, r-1)$:

\begin{definition} \label{2df:hfy}
For $i \in \ZZ_r$, we define a truncated polynomial $\hfy_{\fe_i}$ of $\hUpsilon_{i}$ such that the weight $-\wt$ of $\hUpsilon_i -\hfy_{\fe_i}$ is less than $-\wt(\rfy_{\fe_i})$, i.e., $\hfy_{\fe_i}= \rfy_{\fe_i}+$ certain terms, and the number of the terms is minimal satisfying the relations as $R_\PP$-modules,
$$
\langle \hUpsilon_{1}, \ldots, \hUpsilon_{r-1}\rangle_{R_X}
=\langle \hfy_{\fe_1}, \ldots, \hfy_{\fe_{r-1}}\rangle_{R_X}
=\langle \hfy_{\fe_0}, \ldots, \hfy_{\fe_{r-1}}\rangle_{R_\PP},
\quad \tau_{\tlh}(\hfy_{\fe_i})= \left\{
\begin{matrix} 
1 & \mbox{ for } i=0,\\
0 & \mbox{ otherwise. }
\end{matrix}\right. 
$$
\end{definition}

Since some of $f_{X,y}^{(j)}(P)=0$ at  $P=B_i \in \fB_X\setminus\{\infty\}$, $h_X(x, y_\bullet)\in R_X$ vanishes only at the ramification point $B_i\in X$ and the construction of $h_X$, we have the following lemma from Dedekind's different theorem Proposition 4.27 in \cite{KMP2022a}.

\begin{lemma}\label{2cr:h_X}
$$
\Div(h_X(x,y_\bullet)) = 
\sum_{B_i\in \fB_X\setminus\{\infty\}}d_{B_
i}\ B_i
-d_h\infty,
$$
where $d_{B_i}:=\deg_{B_i,0}(h_x) \ge (e_{B_i}-1)$, and  $\displaystyle{
d_h = \sum_{B_i\in \fB_X\setminus\{\infty\}}d_{B_i}=-\wt (h_X)}$.
\end{lemma}

\begin{definition}\label{2df:fKX_fkX}
The effective divisor, $\displaystyle{
\sum_{B_i\in \fB_X\setminus\{\infty\}}(d_{B_i}-e_{B_i}+1)\ B_i}$, is denoted by
$\fK_X$, i.e. $\fK_X>0$ and let
$\displaystyle{
\fk_X:=\sum_{B_i\in \fB_X\setminus\{\infty\}}(d_{B_i}-e_{B_i}+1)=\deg(\fK_X)
\ge0}$.
\end{definition}

\begin{lemma}\label{2lm:dx_hX}

The divisor of $\displaystyle{\frac{d x}{h_X}}$ is expressed by $(2g-2+\fk_X)\infty-\fK_X$, and $2g-2+\fk_X=d_h - r- 1$ or $\fk_X=d_h - 2g- r+1$.
$\displaystyle{\varphi_\infty\left(\frac{d x}{h_X}\right)=t^{d_h-r-1}dt}$.
\end{lemma}

From Corollary \ref{2cr:(dx)}, we note that these $\fK_X$ and $\fk_X$ play crucial roles in the investigation of the differentials on $X$.

\begin{proposition}\label{2pr:fK_X}
$\fk_X$ is equal to zero if $d_h$ is symmetric whereas $\fk_X$ is not zero otherwise.
\end{proposition}

We recall $\fe_i$ in Definition \ref{2df:NSG1} in the standard basis in Lemma \ref{2lm:Z_standardbasis} and Proposition \ref{2pr:RP-moduleRX}, and $\hfy_{\fe_i}$ in Definition \ref{2df:hfy}.

From Proposition 4.32 in \cite{KMP2022a} we have
the properties of $\displaystyle{\frac{x^k \hfy_{\fe_i}(x,y_\bullet) d x}{h_X(x,y_\bullet)}}$:

\begin{proposition}\label{2pr:phi_hi}
\begin{gather*}
\begin{split}
& \Div\left(\frac{x^k \hfy_{\fe_i}(x,y_\bullet) d x}
 {h_X(x,y_\bullet)}\right) 
=k\, \Div_0(x)+\Div_0(\hfy_{\fe_i})-\fK_X+(\fe_{i}-(k+1)r-1) \infty,
\end{split}
\end{gather*}
where $\Div_0(\hfy_{\fe_i})-\fK_X\ge0$.
We have
$$
\left\{
\wt
\left(\frac{x^k \hfy_{\fe_i}(x,y_\bullet) d x}{h_X(x,y_\bullet)}\right)+1 \
\Bigr| \ i \in \ZZ_r,\ k\in \NN_0\right\}
=\BH_X^\fc.
$$
\end{proposition}

\section{W-normalized Abelian differentials on $X$}
\label{2ssc:W-normAD}

\subsection{W-normalized Abelian differentials  $\bH^0(X, \cA_X(*\infty))$}
\label{2ssc:W-norm_nuI}

Following K. Weierstrass \cite{WeiWIV}, 
H. F. Baker \cite{Baker97}, V.M. Buchstaber, D.V. Leykin and V.Z. Enolskii \cite{BEL20}, J.C. Eilbeck, V.Z. Enolskii and D.V. Leykin \cite{EEL00} and our previous results \cite{KMP13, KMP19, KM2020}, we construct the Abelian differentials of the first kind and the second kind $\bH^0(X, \cA_X(*\infty))$ on $X$ for more general W-curves based on Proposition \ref{2pr:phi_hi} \cite{KMP2022a}.

We consider the Abelian differentials of the first kind on a W-curve.
Due to the Riemann-Roch theorem, there is the $i$-th holomorphic one-form
whose behavior at $\infty$ is given by
\begin{equation}
 \Bigr(t^{N^\fc(g-i)-1} (1+ d_{>0}(t))\Bigr) d t,
\label{2eq:nuI1}
\end{equation}
where $N^\fc(i) \in H_X^\fc$ $(i=1, 2, \ldots, g)$ satisfying $N^\fc(i) <N^\fc(i+1)$, and $t$ is the arithmetic local parameter at $\infty$.
We call this normalization the {\emph{W-normalization}}.
Similarly we find the differentials or the basis of $\bH^0(X, \cA_X(*\infty))$ associated with $\BH_X^\fc$.

The W-normalized holomorphic one-forms are directly obtained from Proposition \ref{2pr:phi_hi}:
\begin{lemma}\label{2lm:hphi_hi}
For $x^k \hfy_{\fe_i}$ in Proposition \ref{2pr:phi_hi}, we have the relation,
$$
\left\langle\frac{x^k\hfy_{\fe_i}}{h_X(x,y_\bullet)}d x \ \Bigr|
\ i \in \ZZ_r^\times, k\in \NN_0,  \right\rangle_\CC
= \bH^0(X, \cA_X(*\infty)).
$$
\end{lemma}

By re-ordering $x^k\hfy_{\fe_j}$ with respect to the weight at $\infty$, we define the ordered set $\{\hphi_i\}$:

\begin{definition}\label{2df:nuI_hX}
\begin{enumerate}
\item 
Let us define the ordered subset
$\hS_X$ of $R_X$ by
$$
\hS_X=\{\hphi_i\ | \ i \in \NN_0\}
$$
such that  $\hphi_i$ is ordered by the Sato-Weierstrass weight,
i.e., $-\wt \hphi_i < -\wt \hphi_j$ for $i < j$,
and 
$\hS_X$ is equal to $\{x ^k \hfy_{\fe_i} \ |\ i \in \ZZ_r,\ k \in \NN_0\}$ as a set.

\item Let $\hR_X$ be an $R_X$-module generated by $\hS_X$, i.e., $\hR_X:=\langle \hS_X\rangle_{R_X} \subset R_X$.

\item Recalling $\fK_X$ and $\fk_X$ in Definition \ref{2df:fKX_fkX}, we let $\hN(n):= -\wt\ (\hphi_n)-\fk_X$, $\hH_X:=\{-\wt\ (\hphi_n)\ | n \in \NN_0\}$, and 
we define the dual conductor $\hc_X$ as the minimal integer satisfying $\hc_X + \NN_0 \subset \hH_X-\fk_X$.

\item We define $\hS_X^{(g)}:=\{\hphi_0, \hphi_1, \ldots, \hphi_{g-1}\}$, and the \emph{W-normalized holomorphic one form}, or \emph{W-normalized Abelian differentials of the first kind} $\nuI{i}$ as  the canonical basis of $X$,
\begin{equation}
\left\langle\nuI{i} := \frac{\hphi_{i-1} d x }{h_X}\ \Bigr|\ 
\hphi_{i-1} \in \hS_X^{(g)}
\right\rangle_\CC = \bH^0(X, \cA_X).
\label{2eq:nuI_hX}
\end{equation}

\end{enumerate}
\end{definition}

We note that  at $\infty$, $\nuI{i}$ behaves like 
$\displaystyle{
\nuI{i}=(t^{N^\fc(g-i-1)-1} (1+ d_{>0}(t))) d t
}$ for the arithmetic local parameter $t$ at $\infty$, and further
$\displaystyle{\frac{\hphi_{i-1} d x }{h_X}=t^{N^\fc(g-i-1)-1}(1+ d_{>0}(t))) d t}$ where $N^\fc(i)$ indicates the element in $\BH_X^\fc$ such that $N^\fc(-i)=-i$ for $i \in \NN$; they are W-normalized Abelian differentials.

We summarize them:
\begin{lemma}\label{2lm:nuI_hX}
\begin{enumerate}

\item 
$\displaystyle{\bH^0(X, \cA_X(*\infty)) = \bigoplus_{i=0} \CC\frac{\hphi_i d x}{h_X}=R_X^\fc d x = \frac{\hR_X d x}{h_X}}$.

\item $\displaystyle{
\lH_X^\fc=
\left\{ \wt\left(\frac{\hphi_i d x}{h_X}\right)+1 \ |\ i \in \NN_0\right\}
=\left\{ \wt(\hphi_i)+d_h -r \ |\ i \in \NN_0\right\}
}$

\hskip 15pt
$=\left\{ \wt(\hphi_i)+2g-1 -\fk_X \ |\ i \in \NN_0\right\}
$.

\end{enumerate}
\end{lemma}

\bigskip

By the Abel-Jacobi theorem \cite{FarkasKra}, $\fK_X$ in Definition \ref{2df:fKX_fkX} can be divided into two pieces, which are related to the spin structure in $X$.

\begin{definition}\label{2df:fK_s}
Let $\fK_\fs$ and $\fK_X^\fc$ be the effective divisors which satisfy
$$
\fK_X-\fk_X\infty \sim2\fK_\fs-2 \fk_\fs\infty,\quad
\fK_X + \fK_X^\fc - (\fk_X+\fk_X^\fc)\infty\sim 0
$$
as the linear equivalence, where $\fk_\fs$ and $\fk_X^\fc$ are the degree of $\fK_\fs$ and $\fK_X^\fc$ respectively.
\end{definition}

\bigskip

Since the W-normalized holomorphic one form is given by the basis (\ref{2eq:nuI_hX}), Definition \ref{2df:fK_s} shows the canonical divisor:

\begin{proposition}\label{2pr:cKX}
The canonical divisor is given by
\begin{gather*}
\begin{split}
K_X \sim(2g-2+\fk_X)\infty-\fK_X&\sim (2g-2+2\fk_\fs)\infty - 2\fK_\fs\\
&\sim (2g-2-\fk_X^\fc)\infty +\fK_X^\fc.
\end{split}
\end{gather*}
\end{proposition}

From \cite[Lemmas 5.8 and 5.9]{KMP2022a}, we show the properties of these parameters:

\begin{lemma}\label{2lm:nuIo1_2}
\begin{enumerate}

\item $\{-\wt(\hphi_i)\}=\{d_h - \fe_i + k r\ |\ i \in \ZZ_r,\ k \in \NN_0\}$.

\item $\Div(\hphi_i)\ge(\fK_X-(2g-2+\fk_X)\infty)$ for every $\hphi_i 
\in \hS_R^{(g)}$, $(i=0, 1, 2,\ldots,g-1)$.

\item $\Div(\hphi_i)\ge(\fK_X-(g+\fk_X+i)\infty)$ for every $\hphi_i
\in \hS_R$, $(i\ge g)$.
\end{enumerate}
\end{lemma}

\begin{lemma}\label{2lm:nuIo2}
\begin{enumerate}
\item $-\wt\ \hphi_{0}= (0$ if $H_X$ is symmetric) 
$= \hc_X+\fk_X-c_X
=d_h - r-c_X+1$,

\item $-\wt\ \hphi_{g-1}= \hc_X+\fk_X-2= d_h - r- 1=
(2g-2)+\fk_X$, i.e., $\hN(g-1) = 2g-2$,

\item $ \hc_X = 2g=d_h -\fk_X-r+1$, 
$-\wt\ \hphi_{g}=2g+\fk_X=\hc_X+\fk_X=d_h - r+1$,

\item $-\wt\ \hphi_{g-1}+\wt\ \hphi_{0}=\fe_{r-1}-r-1= c_X-2$, and
$c_X= \fe_{r-1} - r+1$.

\end{enumerate}
\end{lemma}

\subsection{Extension $p\in R_X\otimes_{R_\PP} R_X$ to $R_X\otimes_{\CC}R_X$}
\label{2ssec:Ext_p}

As we have $\tlh_{R_X}$ in Lemma \ref{2lm:h_RXe}, we extend it in $R_X\otimes_{R_\PP} R_X$ to an element in $R_X\otimes_{\CC} R_X$, though the extension is not unique; there are two different $h(x,y_\bullet, x', y_\bullet')$ and $h'(x,y_\bullet, x', y_\bullet')$ in $R_X\otimes_{\CC} R_X$ such that $h(x,y_\bullet, x, y_\bullet')=h'(x,y_\bullet, x, y_\bullet')$ in $R_X\otimes_{R_\PP} R_X$.
Since they are quasi-isomorphic, we select one of them, and thus, it is well-defined in the meaning of Proposition \ref{2pr:Sigma} 4. and \ref{2pr:dSigma1} 4.

\begin{definition}\label{2df:th_X}
Using $\hUpsilon_{i}$ in Proposition \ref{2pr:h_RXe_ys}, for a point $(P = (x, y_\bullet), P'=(x', y_\bullet'))\in X \times X$, we define a polynomial $\tlh_X(x, y_{\bullet}, y_{\bullet}')\in R_X \otimes_\CC R_X$ by
$$
\tlh_X(x, y_{\bullet}, y_{\bullet}')
:= \hUpsilon_{0}\cdot 1
+ \hUpsilon_{1}\fy'_{\fe_1}+\cdots+ \hUpsilon_{r-1}\fy'_{\fe_{r-1}}.
$$
\end{definition}

Then the following lemma is evident from Lemma \ref{2lm:h_RXe}.

\begin{lemma}\label{2lm:h_X}
For a point $(P = (x, y_\bullet), P'=(x', y_\bullet'))\in X \times X$, $\tlh_X(x, y_{\bullet}, y_{\bullet}')$ satisfies
\begin{enumerate}
\item for the case $x'=x$ and a group action $\hzeta \in G_x$,
$\displaystyle{
 \frac{\tlh_X(x, \hzeta y_{\bullet},\hzeta y_{\bullet}')}
{\tlh_X(x, \hzeta y_{\bullet},\hzeta y_{\bullet})}=
 \frac{\tlh_X(x, y_{\bullet}, y_{\bullet}')}
{\tlh_X(x, y_{\bullet},y_{\bullet})}
}$,

\item 
$\displaystyle{
\lim_{(x,y) \to \infty} \frac{\tlh_X(x, y_{\bullet}, y_{\bullet}')}
{\tlh_X(x, y_{\bullet},y_{\bullet})}=\frac{1}{r}}$ for $P'\in X\setminus\{\infty\}$, 

\item when $x=x'$ or $\tlh_X(x, y_{\bullet},x', y_{\bullet}')$ belongs to $R_X \otimes_{R_\PP} R_X$, it satisfies
$$
\tlh_X(x, y_{\bullet}, y_{\bullet}')
=\tlh_{R_X}(x, y_{\bullet}, y_{\bullet}'),
$$

\item $h_X(x, y_{\bullet})=\tlh_X(x, y_{\bullet}, y_{\bullet})$.
\end{enumerate}
\end{lemma}

\begin{definition}\label{2df:p_varpi}
Using $\tlh_X(x, y_{\bullet}, y_{\bullet}')$ for a point $(P=(x, y_{\bullet}), P'=(x', y_{\bullet}')) \in X \times X$, we define
$$
p_{\varpi}(P,P')
:=\frac{\tlh_X(x, y_{\bullet}, y_{\bullet}')}
{\tlh_X(x, y_{\bullet})}.
$$
\end{definition}

It is obviously that $p_{\varpi}(P,Q)$ belongs to $\fh_{X,P} R_{X,P}$ as a function of $P$ at $P\in X$ and thus Proposition \ref{2pr:phi_hi} whose origin is the Dedekind's different theorem \cite[Proposition 4.27]{KMP2022a} shows the proposition:
\begin{proposition}\label{2pr:DedekindDiff3}
$e_{B_i}-1=-\deg_{P=B_i}(p_{\varpi}(P,Q))$  for $Q\in X\setminus\{\infty\}$.
\end{proposition}

Further the direct computations provide the following proposition:

\begin{proposition}\label{2pr:p_varpi}
For $Q \in X\setminus\{\infty\}$, 
$$
p_{\varpi}(P,Q)
=\left\{
\begin{array}{rcl}
1 & \mbox{  for  } &P = Q,\\
0 & \mbox{  for  } & P \neq Q \mbox{ and } 
    \varpi_x(P)=\varpi_x(Q),
\end{array}\right. 
$$
$$
\lim_{P\to \infty} p_{\varpi}(P,Q)=\frac{1}{r}, \quad
\lim_{Q\to \infty} p_{\varpi}(P,Q)=\infty.
$$
\end{proposition}

\begin{remark}{\rm{
We should remark that $p(P,Q)\in R_X\otimes_{\CC} R_X$ which holds the relations in Proposition \ref{2pr:p_varpi} is not unique.
The problem comes from the fact that there are infinitely many different $p'(P,Q)$ from $p(P,Q)$  such that
$
p(P,Q)=p'(P,Q)
$ 
for $\varpi_r(P)=\varpi_r(Q)$.

}}
\end{remark}

\begin{proposition}\label{2pr:p_varpi2}
For a point $Q\in X\setminus \fB_X$,
$$
       \Div_P(p_{\varpi}(P,Q)) =
 \sum_{\zeta \in G_{X, \varpi_x(Q)}, \zeta \neq e} \zeta Q
+ E_Q
                 -\sum_{i=1}^{\ell_{\fB}} (e_{B_i}-1) B_i,
$$
where $E_Q$ is an effective divisor  such that $\deg\  E_Q=\displaystyle{
\left(\sum_{i=1}^{\ell_\fB} (e_{B_i} -1 ) -r+1\right)=2g}$.

Further, this relation is extended to the condition $Q \in \fB_X\setminus\{\infty\}$ by considering the multiplicity of the action $G_{\varpi_r(Q)}$.
\end{proposition}

\begin{proof}
At the ramification point $B_i$ of $\varpi_x: X \to \PP$, Proposition \ref{2pr:DedekindDiff3} shows the third term.
We note that $p_{\varpi}(\infty,Q)=1/r$.
From the Riemann-Hurwitz theorem (\ref{2eq:RHt_wr}), there exist the first and the second terms.\qed
\end{proof}

For an element $f\otimes f'$ in $R_X\otimes_{\CC}R_X$, we define the
weight $\wt(f\otimes f')$ by $\wt(f)+\wt(f')$.

\begin{lemma}\label{2lm:h_Xhomo}
$\tlh_X(x, y_\bullet, x', y'_\bullet)$ is a homogeneous element in $R_X\otimes_{\CC}R_X$  whose weight is $d_h$ or $-\wt \ \tlh_X(x, y_\bullet, x', y'_\bullet)=-\wt\ h_X(x, y_\bullet)=-d_h$ and thus the weight of $p_\varpi$ is zero in $\cQ(R_X\otimes_{\CC}R_X)$.
\end{lemma}

For later convenience, we introduce $\varphi_\infty(R_X\otimes_{\CC} R_X)$.

\begin{definition}\label{2df:th_H}
We define a polynomial in $R_H^Z \otimes_\CC R_H^Z$ such that
$$
\tlh_H(Z, Z'):=\varphi_H^X(\tlh_X(x,y_\bullet, x',y_\bullet'))
=:\sum_{i=0}^{r-1}
Z_r^{\delta_i} \fZ_{\fe_{\ell,i}^*} \fZ_{\fe_i}',
$$
and an element,
$$
p_H(Z, Z'):=
\frac{\tlh_H(Z, Z')}
{h_H(Z_r, Z_\bullet)} \in 
\cQ(R_H^Z \otimes_\CC R_H^Z).
$$
\end{definition}

This $\tlh_H(Z, Z')$ is uniquely defined in the meaning of Proposition \ref{2pr:Sigma} 4. and \ref{2pr:dSigma1} 4.

\begin{proposition}\label{2pr:p_H}
When $Z_r=Z_r'$,
$\tlh_H(Z, Z')$ and 
$p_H(Z, Z')$ agree with
$\tlh_{H}(Z_r, Z_\bullet, Z'_\bullet)$ in Lemma \ref{2lm:h_RXe} and 
$p_{H}(Z_r, Z_\bullet,  Z'_\bullet)$ in Lemma 
\ref{2lm:p_HZ_r} respectively.
$\displaystyle{\varphi_H^Z(\tlh_H(Z, Z'))=z^{d_h}\sum_{i\in \ZZ_r} 
\frac{z^{\prime\, \fe_i}}{z^{\fe_i}}}$.
\end{proposition}

\subsection{W-normalized differentials of the second and the third kinds}
\label{2ssc:Wdiff23}

\subsubsection{The one-form $\Sigma$ on $X$}

We construct an algebraic representation of the fundamental W-normalized differential of the second kind in \cite[Corollary 2.6]{Fay73}, namely, a two-form $\Omega(P_1, P_2)$ on $X\times X$ which is symmetric and has quadratic singularity as in Theorem \ref{2th:Sigma}.

Following K. Weierstrass \cite{WeiWIV}, H. F. Baker \cite{Baker97}, V.M. Buchstaber, D.V. Leykin and V.Z. Enolskii [\cite{BEL20} and therein], J.C. Eilbeck, V.Z. Enolskii, and D.V. Leykin \cite{EEL00}, we have $\Omega$ using a meromorphic one-form $\Sigma(P, Q)$ on $X \times X$ for the hyperelliptic curves and plane Weierstrass curves (W-curves).
In this subsection, we extend it to more general W-curves based on Definition \ref{2df:p_varpi} and Proposition \ref{2pr:p_varpi} to introduce $\Sigma(P, Q)$ on $X \times X$.

\begin{proposition}\label{2pr:Sigma}
For a point $(P, Q) \in X \times X$,
\begin{equation}
   \Sigma\big(P, Q\big)
   :=
\frac{d x_P}{(x_P - x_Q)}p_{\varpi}(P,Q) 
=
\frac{d x_P}{(x_P - x_Q)}
\frac{\tlh_X(x_P, y_{\bullet P},  y_{\bullet Q})}
{h_X(x_P, y_{\bullet P})}
\label{2eq:Sigma}
\end{equation}
 has the following properties:
\begin{enumerate}
\item 
For a group action $\hzeta \in G_{\varpi_r(P)}$, 
$\Sigma(\hzeta P, \hzeta Q)=\Sigma(P, Q)$ if $\varpi_r(P)=\varpi_r(Q)$.

\item $\Sigma(P, Q)$ is holomorphic over $X$ except $Q$ and $\infty$ as a function of $P$ such that
\begin{enumerate}
\item at $Q$, in terms of 
the local parameters $t_Q(P)=0$ and $t_P(P)=0$, it behaves like
$$
\Sigma\big(P, Q\big)=\frac{d t_P}{t_P}(1 + d_{>0}(t_P,t_Q)),
$$

\item at $\infty$, the local parameter $t_P$ $(t_P(\infty)=0)$, it behaves like
$$
\Sigma\big(P, Q\big)=-\frac{d t_P}{t_P}(1 + d_{>0}(t_P)).
$$
\end{enumerate}

\item $\Sigma(P, Q)$ as a function of $Q$ is singular at $P$ and $\infty$ such that
\begin{enumerate}
\item at $P$, in terms of the local parameters $t_Q(P)=0$ and $t_P(P)=0$, it behaves like
$$
\Sigma\big(P, Q\big)=-\frac{d t_P}{t_Q}(1 + d_{>0}(t_P,t_Q)),
$$

\item at $\infty$, the arithmetic local parameter $t_Q$, $(t_Q(\infty)=0)$, it behaves like
$$
\Sigma\big(P, Q\big)=\frac{d x_P}{\fy_{\fe_{r-1} P}t_Q^{\fe_{r-1}-r}}
(1 + d_{>0}(t_Q))=
\frac{d x_P}{\fy_{\fe_{r-1} P}t_Q^{c_X-1}}
(1 + d_{>0}(t_Q)).
$$
\end{enumerate}

\item Let $\tlh_X'(x_P, y_{\bullet P},  y_{\bullet Q})$ be an element in $R_X \otimes_\CC R_X$ satisfying the conditions in Lemma \ref{2lm:h_X}, i.e., $\tlh_X'(x_P, y_{\bullet P},  y_{\bullet Q})=\tlh_X(x_P, y_{\bullet P},  y_{\bullet Q})$ at $\varpi_x(P)=\varpi_x(Q)$, and let
$$
\Sigma'\big(P, Q\big)
   :=
\frac{d x_P}{(x_P - x_Q)}
\frac{\tlh_X'(x_P, y_{\bullet P},  y_{\bullet Q})}
{h_X'(x_P, y_{\bullet P})}.
$$
Then $\Sigma'\big(P, Q\big)-\Sigma\big(P, Q\big)$ belongs to the set,
$$
            \bH^0(X, \cA_X)\otimes_\CC \bH^0(X, \cO_X(*\infty))
               =\bH^0(X, \cA_X)\otimes_\CC R_X.
$$.
\end{enumerate}
\end{proposition}

\begin{proof}
Lemma \ref{2lm:h_X} 1 shows {\it{1}}.
Noting the properties in Proposition \ref{2pr:p_varpi2} and Corollary \ref{2cr:(dx)}, the numerator of $\Sigma$ is zero with the first order at the points which $P \neq Q$ and $P\in \varpi_x^{-1}(\varpi_x(Q))$ and thus, $\Sigma$ behaves like finite one-form there.
At $P=Q$, the numerator is equal to one and thus we have $\Sigma=d x_P/(x_P - x_Q)$, which means $\Sigma=d t_P/t_P (1+d_{>0}(t_P))$ in {\it{2.a}} and {\it{3.a}}.

Recalling $\varphi_\infty:R_X \to R_H$ in Lemma \ref{2lm:varphiinf}, let $t_P$ and $t_Q$ be the local parameters at $\infty$ corresponding to $P$ and $Q$ respectively.
Proposition \ref{2pr:p_varpi} shows $\displaystyle{
\lim_{P\to \infty} p_{\varpi}(P,Q)=\frac{1}{r}}$
whereas 
$\displaystyle{
\lim_{P\to \infty} \frac{d x_P}{(x_P-x_Q)} = - r(1+d_{>0}(t_P)) d t_P}$,
and thus we have {\it{2.b}}.
On the other hand, the following Lemma \ref{2lm:SigmaH1} 3 shows that $\displaystyle{\Sigma(P,Q)}$ behaves 
$\displaystyle{\sum_{k=1}^\infty
\sum_{i=0}^{r-1}\frac{
t_P^{\fe_i-r k}}{t_Q^{\fe_i-r k}}\frac{d t_P}{t_P}}$
\break
$\displaystyle{
(1+d_{>0}(t_P))
}$.
The maximum of $\{\fe_i-r\}_{i\in \ZZ_r}$ is $\fe_{r-1}-r$, which is equal to $c_X-1$ due to Lemma \ref{2lm:nuIo2} 4, and thus  $\displaystyle{\varphi_\infty
\left(\frac{d x}{\fy_{\fe_{r-1}}}\right)=
\frac{t_P^{\fe_{r-1}}}{rt_P^{r+1}}d t_P=\frac{t_P^{c_X-2}d t_P}{r}}$.
We obtain {\it{3.b}}.

Let us consider {\it{4}}.
Since both $\Sigma'\big(P, Q\big)$ and $\Sigma\big(P, Q\big)$ satisfy these properties 1-3, their difference is holomorphic over $X$ with respect to $P$, and over $X\setminus \{\infty\}$ with respect to $Q$.
It shows {\it{4}}.
\qed
\end{proof}

\subsubsection{The one-form $\Sigma$ at $\infty$}

Noting Lemma \ref{2lm:varphiinf}, we consider a derivation in $\cQ(R_H^Z\otimes R_H^Z)$ for the monomial curve $X_H\times X_H$ instead of $R_{X, \infty}$ using surjection $\varphi_\infty$.

In order to investigate the behavior of $\Sigma$ at $\infty$, we consider the differential in monomial curve $X_H\times X_H$ and element in $\cQ(R_H^Z \otimes_\CC R_H^Z)$ (rather than $\cQ(R_H^Z \otimes_{\CC[Z_r]} R_H^Z)$),
\begin{equation}
\frac{d Z_r}{(Z_r - Z_r')}
p_{H}(Z_r, Z_\bullet, Z'_\bullet).
\label{2eq:Sigma_X^H}
\end{equation}
Noting $\displaystyle{
p_{H}(Z_r, Z_\bullet, Z'_\bullet))
=\frac{\tlh_H(Z_r, Z_\bullet, Z'_\bullet)}{
\tlh_H(Z_r, Z_\bullet, Z_\bullet)}
}$ and Lemmas \ref{2lm:varphiinf},
\ref{2lm:p_HZ_r}, and \ref{2lm:h_Xhomo},
we define
\begin{equation}
 \Sigma_H\big(t_P, t_Q\big)
:=
-\frac{t_Q^r}{t_Q^r-t_P^r}
\sum_{i=0}^{r-1}\frac{
t_P^{\fe_i}}{t_Q^{\fe_i}}\frac{d t_P}{t_P}
\label{2eq:Sigma_X^Ht}
\end{equation}
as an element in $\cQ(R_H \otimes_\CC R_H)$ using isomorphism $\varphi_H^Z$ and the parameters $t_P=1/z_P$ and $t_Q=1/z_Q$;
$\Sigma_H$ is regard as a derivation in $\cQ(R_{X,\infty}\otimes_\CC R_{X,\infty})$. 

Then the direct computations lead the following results:
\begin{lemma}\label{2lm:SigmaH1}
\begin{enumerate}
\item 
$\displaystyle{
\varphi_\infty(\Sigma(P, Q)) =
 \Sigma_H\big(t_P, t_Q\big)}$
with the local parameters $t_P$ and $t_Q$ at $\infty$. 
and it shows the behavior at $\infty$, i.e.,
$$
\Sigma(P, Q) =
 \Sigma_H\big(t_P, t_Q\big)(1+d_{>0}(t_P, t_Q))
$$
for the local parameters at $\infty$,
$t_P(\infty)=t_Q(\infty) = 0$.

\item for the case $|t_P|<|t_Q|$,
$\displaystyle{
\Sigma_H\big(t_P, t_Q\big)
=
-
\sum_{k=0}^\infty
\sum_{i=0}^{r-1}\frac{
t_P^{\fe_i+r k}}{t_Q^{\fe_i+r k}}\frac{d t_P}{t_P},
}$ and

\item for the case $|t_Q|<|t_P|$,
$\displaystyle{
\Sigma_H\big(t_P, t_Q\big)
=
\sum_{k=1}^\infty
\sum_{i=0}^{r-1}\frac{
t_P^{\fe_i-r k}}{t_Q^{\fe_i-r k}}\frac{d t_P}{t_P}
}$, and $|\fe_i - k r|\neq 0$, $k>0$.
\end{enumerate}
\end{lemma}

\begin{proof}
{\it{1}} is obvious. The relations
$\displaystyle{
\frac{t_Q^r}{t_Q^r-t_P^r}
=\frac{1}{1-(t_P/t_Q)^r}=\sum_{k=0}^\infty
\left(\frac{t_P}{t_Q}\right)^{r k}}$, and 
$\displaystyle{
\frac{t_Q^r}{t_Q^r-t_P^r}
}$
$\displaystyle{
=-\frac{t_Q^r}{t_P^r}\frac{1}{1-(t_Q/t_P)^r}=}$
$\displaystyle{
-\frac{t_Q^r}{t_P^r}\sum_{k=0}^\infty
\left(\frac{t_Q}{t_P}\right)^{r k}}$ show {\it{2}} and {\it{3}}.
\qed
\end{proof}

\subsubsection{The differential $d\Sigma$ on $X$}

In order to define $\Omega$, we consider the derivative $d\Sigma$ in this subsubsection.

\begin{proposition} \label{2pr:dSigma1}
\begin{equation}
 d_{Q} \Sigma\big(P, Q\big)
   :=d x_Q\frac{\partial }{ \partial x_Q}
 \Sigma\big(P, Q\big)
\end{equation}
has the following properties:
\begin{enumerate}
\item $d_{Q} \Sigma\big(P, Q\big)$ as a function of $P$ is holomorphic over $X$ except $Q$ such that
\begin{enumerate}
\item at $Q$, in terms of  the local parameter $t_P(Q)=0$, it behaves like
$$
d_{Q} \Sigma\big(P, Q\big)=\frac{d t_P d t_Q}{t_P^2}(1 + d_{>0}(t_P)),
$$

\item at $\infty$, in terms of the arithmetic parameters $t_P$ and $t_Q$, it behaves like
$$
d_{Q} \Sigma\big(P, Q\big)= r
\frac{t_P^{r-1}}{t_Q^{r+1}}
d t_P d t_Q(1 + d_{>0}(t_P, t_Q)),
$$
and $d_{Q} \Sigma\big(P, Q\big)$ is holomorphic at $\infty$.
\end{enumerate}

\item $d_{Q} \Sigma\big(P, Q\big)$ as a function of $Q$ is holomorphic over $X$ except $P$ and $\infty$ such that
\begin{enumerate}
\item at $P$, in terms of the local parameter $t_Q(P)=0$, it behaves like
$$
d_{Q} \Sigma\big(P, Q\big)=\frac{d t_P d t_Q}{t_Q^2}(1 + d_{>0}(t_Q)),
$$

\item at $\infty$, using the arithmetic parameter $t_Q(\infty)=0$, it behaves like
$$
d_{Q} \Sigma\big(P, Q\big)= \frac{1}{t_Q^{c_X}}
d x_P d t_Q(1 + d_{>0}(t_Q)).
$$
\end{enumerate}
\item We have the following relations,
\begin{enumerate}
\item 
$\displaystyle{
d_{Q} \Sigma\big(P, Q\big)- d_{P} \Sigma\big(Q, P\big)\in
\bH^0(X, \cA_X(*\infty))^{2\otimes_\CC}
}$,

\item 
$\displaystyle{  \oint_{Q=\infty} \Sigma(P,Q) = 0 }$,

\item 
$\displaystyle{ \oint_\infty \left(\int^Q d_{Q'}\Sigma(P,Q)\right)
= \oint_{P=\infty} \Sigma(P,Q) = 2\pi \ii }$, and 

\end{enumerate}

\item For $\Sigma'\big(P, Q\big)$ defined in Proposition \ref{2pr:Sigma} 4, we have
$$
d_Q\Sigma'\big(P, Q\big)-d_Q\Sigma\big(P, Q\big)\in
\bH^0(X, \cA_X)\otimes_\CC d\bH^0(X, \cO_X(*\infty))
= \bH^0(X, \cA_X)\otimes_\CC d R_X.
$$.

\end{enumerate}
\end{proposition}

\begin{proof}
{\it{1.a}}, {\it{2.a}}, and finite part of {\it{3.a}} are directly obtained from Proposition \ref{2pr:Sigma}.
The properties of $d\Sigma$ on $X_H\times X_H$ determines the behavior at $\infty$ in the following subsection, Subsection \ref{2ssc:dSigmaH}.

We use the facts: $\displaystyle{d x_Q\frac{\partial}{\partial x_Q}
= d t_Q\frac{\partial}{\partial t_Q}}$, and
$\displaystyle{d t_Q\frac{\partial}{\partial t_Q} t_Q^\ell=\ell\,  t_Q^{\ell-1}d t_Q }$.
Lemmas \ref{2lm:dSigmaH1} and \ref{2lm:SigmaH1} show {\it{1.b}}, {\it{2.b}} and the behavior at $\infty$ of {\it{3.a}}. Lemma \ref{2lm:SigmaH1} 3 shows {\it{3.b}}.

Lemma \ref{2lm:Deltair} shows {\it{3.c}}, and Proposition \ref{2pr:Sigma} 4 shows {\it{4}}. 
\qed
\end{proof}

It means that there exist $F_{d\Sigma}(P, Q)\in R_X\otimes_\CC R_X$ such that
\begin{equation}
d_Q\Sigma(P,Q)= \frac{F_{d\Sigma}(P, Q)d x_P \otimes d x_Q}
{(x_P-x_Q)^2 h_X(x_{P},y_{\bullet P})h_X(x_{Q},y_{\bullet Q})}.
\label{2eq:Kform'}
\end{equation}

\subsubsection{The differential $d\Sigma_H$ on $X_H$}\label{2ssc:dSigmaH}

Noting Lemma \ref{2lm:varphiinf}, we also consider $d\Sigma$ as a two-form in $\cQ(R_H^Z\otimes_\CC R_H^Z)$ and the monomial curve $X_H\times X_H$ instead of $R_{X, \infty}$.

We consider the differential of (\ref{2eq:Sigma_X^H}),
$$
d\Sigma_H= d Z_r'
\frac{\partial }{\partial Z_r'}\frac{d Z_r}{(Z_r - Z_r')}
\frac{\tlh_H(Z_r, Z_\bullet, Z'_\bullet)}{
\tlh_H(Z_r, Z_\bullet, Z_\bullet)}
$$
which is equal to
$$
\frac{d Z_r}{(Z_r - Z_r')^2}
\frac{\tlh_H(Z_r, Z_\bullet, Z'_\bullet)}{
\tlh_H(Z_r, Z_\bullet, Z_\bullet)}
+
\frac{d Z_r}{(Z_r - Z_r')\tlh_H(Z_r, Z_\bullet, Z_\bullet)}
\frac{\partial }{\partial Z_r'}\tlh_H
(Z_r, Z_\bullet, Z'_\bullet).
$$
We recall the correspondence between $z$ of $R_H$ and $t$ of  $R_{X,\infty}$ in Lemma \ref{2lm:varphiinf}.
Using (\ref{2eq:Sigma_X^Ht}), we have
$$
d_Q \Sigma_H\big(t_P, t_Q\big)
=-d t_Q \frac{\partial }{\partial t_Q}
\frac{t_Q^r}{t_Q^r-t_P^r}
\sum_{i=0}^{r-1}\frac{
t_P^{\fe_i}}{t_Q^{\fe_i}}
\frac{d t_P}{t_P}.
$$

\begin{lemma}\label{2lm:dSigmaH1}
\begin{enumerate}

\item We have the equality:
\begin{gather*}
d_Q \Sigma_H\big(t_P, t_Q\big)=
\frac{d t_P d t_Q}
{(t_Q^r-t_P^r)^2t_Pt_Q}
\left[
\sum_{i=0}^{r-1}\left(
\fe_i\frac{
t_P^{\fe_i}}{t_Q^{\fe_i-2r}}
-
(\fe_i-r)\frac{
t_P^{\fe_i+r}}{t_Q^{\fe_i-r}}
\right)
\right].
\end{gather*}

\item For the case $|t_Q|<|t_P|$, we have the expansion,
\begin{equation}
d_Q \Sigma_H\big(t_P, t_Q\big)=d t_P d t_Q
\left[
\sum_{i=0}^{r-1}
\sum_{k=1}^{\infty}
\left(
-(\fe_i-k r)\frac{
t_P^{\fe_i-r k-1}}{t_Q^{\fe_i-r k+1}}
\right)
\right].
\label{2eq:dQSigmaH11}
\end{equation}

\item For the case $|t_P|<|t_Q|$, we have the expansion,
\begin{equation}
d_Q \Sigma_H\big(t_P, t_Q\big)=d t_P d t_Q
\left[
\sum_{i=0}^{r-1}
\sum_{k=1}^{\infty}
\left(
(\fe_i+k r)\frac{
t_P^{\fe_i+r k-1}}{t_Q^{\fe_i+r k+1}}
\right)
\right].
\label{2eq:dQSigmaH12}
\end{equation}
\end{enumerate}
\end{lemma}

\begin{proof}
{\it{1}}: Using the relation, $\displaystyle{
\frac{\partial }{\partial t_Q}
\frac{t_Q^r}{t_Q^r-t_P^r}=
\left(r-\frac{rt_Q^r}{t_Q^r-t_P^r}\right)\frac{t_Q^{r-1}}{t_Q^r-t_P^r}
=-\frac{rt_Q^r t_P^r}{(t_Q^r-t_P^r)^2 t_Q}
}$, we have
\begin{gather*}
d_Q \Sigma_H\big(t_P, t_Q\big)=
\left[\left(
\frac{r t_P^r}{t_Q^r-t_P^r }\right)
\sum_{i=0}^{r-1}\frac{
t_P^{\fe_i}}{t_Q^{\fe_i-r}}
+
\sum_{i=0}^{r-1}\fe_i\frac{
t_P^{\fe_i}}{t_Q^{\fe_i-r}}
\right]
\frac{1}{t_Q^r-t_P^r}
\frac{d t_P d t_Q}{t_Pt_Q}.
\end{gather*}
The relation
$r t_P^r+\fe_i(t_Q^r- t_P^r)=
\fe_i t_Q^r - (\fe_i - r)t_P^r$ enables us to obtain {\it{1}}.

{\it{2}}: Since
 $\displaystyle{
\frac{1}{(1-x)^2} = \frac{d}{d x}\frac{1}{(1-x)}=\sum_{k=1}k x^{k-1}}$, 
for a polynomial $H(t_P,t_Q)\in \CC[t_P,t_Q]$, we have
$$
\frac{H(t_P,t_Q)}{(t_Q^r-t_P^r)^2 }
=\frac{H(t_P,t_Q)}{t_P^{2r}}\sum_{k=1}^{\infty}k 
\frac{t_Q^{r(k-1)}}{t_P^{r(k-1)}}
=\frac{H(t_P,t_Q)}{t_P^{2r}}
\left(\sum_{k=1}^{N}k 
\frac{t_Q^{r(k-1)}}{t_P^{r(k-1)}}
+\sum_{k=N+1}^{\infty}k 
\frac{t_Q^{r(k-1)}}{t_P^{r(k-1)}}\right).
$$
and
\begin{gather*}
d_Q \Sigma_H\big(t_P, t_Q\big)=d t_P d t_Q
\left[
\sum_{i=0}^{r-1}
\sum_{k=1}^{\infty}
\left(
k\fe_i\frac{
t_P^{\fe_i-r(k+1)-1}}{t_Q^{\fe_i-r(k+1)+1}}
-
k(\fe_i-r)\frac{
t_P^{\fe_i-r k-1}}{t_Q^{\fe_i-r k+1}}.
\right)\right].
\end{gather*}
The relation $(k-1)\fe_i - k (\fe_i -r) = -(\fe_i - k r)$ shows {\it{2}}.
Similarly, we prove {\it{3}}.
\qed
\end{proof}

\begin{lemma}\label{2lm:dSigma_H22}
 $d_Q \Sigma_H\big(t_P, t_Q\big)- d_P \Sigma_H\big(t_Q, t_P\big)$ is equal to
\begin{gather*}
d t_P d t_Q\sum_{i=0}^{r-1}
\frac{
\fe_i(t_P^{2\fe_i - r} t_Q^r-
t_Q^{2\fe_i - r} t_P^r)
-
(\fe_i-r)(t_P^{2\fe_i}-
t_Q^{2\fe_i})}
{(t_Q^r -t_P^r) t_P^{\fe_i-r+1}t_Q^{\fe_i-r+1}}.
\end{gather*}
When $t_P=t_Q+\varepsilon$,  it vanishes for the limit $\varepsilon \to 0$.
\end{lemma}

\begin{proof}
The direct computations show them.
\qed
\end{proof}

\bigskip

For the case $|t_Q|<|t_P|$, noting
$\varphi_\infty(h_X(P))=r\, t_P^{-d_h}$, and 
$\varphi_\infty(d x_P)=r\, t_P^{-r-1}d t_P$,
(\ref{2eq:dQSigmaH11}) in Lemma \ref{2lm:dSigmaH1}
leads the expression,
\begin{equation}
d_Q \Sigma_H\big(t_P, t_Q\big)=
\frac{d t_P}{t_P^{r+1}} \frac{d t_Q}{t_Q^{r+1}}
\left[
\sum_{i=0}^{r-1}
\sum_{k=1}^{\infty}
\left(
-(\fe_i-k r)t_P^{d_h}t_Q^{d_h}
\frac{
t_P^{\fe_i-d_h-r(k-1)}}{t_Q^{\fe_i+d_h - r(k+1)}}
\right)
\right].
\label{2eq:dQSigmaH11b}
\end{equation}

\begin{lemma}\label{2lm:hH^g}
Letting $\hH^{(g)}:=\{d_h-\fe_i+kr\ |\ i\in \ZZ_r, k\in \NN_0, \fe_i-(k+1)r >0\}$, and 
$\hH^{(g)*}_X:=\{d_h + \fe_i - r (k+2) \ |\ 
d_h - \fe_i + r k\in \hH^{(g)}\}$, we have
$$
\hH^{(g)}_X\subset\hH_X, \quad
\hH^{(g)*}_X\subset\hH_X,
\quad
\# \hH^{(g)}_X=g, \quad\# \hH^{(g)*}_X=g.
$$
\end{lemma}

\begin{proof}
$\hH^{(g)}$ corresponds to the weight of $\hphi_i$ consisting of $\bH^0(X, \cA_X)$ and thus $\hH^{(g)}_X\subset\hH_X$ and its cardinality is $g$.
On the other hand, the condition in $\hH^{(g)*}$ means
that $d_h + \fe_i - r(k+2) \ge (d_h -r+ 1)+(\fe_i - r)
>\hc_X+\fk_X$
because of $\hc_X+\fk_X-2 = d_h -r -1$ from Lemma \ref{2lm:nuIo2}.
Thus $\hH^{(g)*}\subset\hH_X$, and $\# \hH^{(g)*}_X=g$
\qed
\end{proof}

Since the order of the singularity of $\Sigma_H(t_P, t_Q)$ at $t_Q=\infty$ is $c_X-1=\fe_{r-1}-r$, we introduce the two-form,
$$
\Omega_H^{(0)}(t_P, t_Q):=d t_P d t_Q
\left[
\sum_{i=0}^{r-1}
\sum_{\substack{k=1\\
 d_h - \fe_i + r k \in \hH^{(g)}
}}
\left(
(\fe_i-k r)\frac{
t_P^{\fe_i-r k-1}}{t_Q^{\fe_i-r k+1}}
\right)
\right].
$$

\begin{lemma}\label{2lm:OmegaH2}
$\Omega_H^{(1)}(t_P, t_Q):=d\Sigma_H(t_P, t_Q)-\Omega_H^{(0)}(t_P, t_Q)$ has no singularity at $t_Q=\infty$.
\end{lemma}

\begin{lemma}\label{2lm:Deltair}
The each term in $\Omega_H^{(0)}(t_P, t_Q)$ has the property,
$$
\oint_{t_Q=0}
\left[\int_0^{t_Q} -(\fe_i-k r) 
\frac{t_P^{\fe_i-k r-1}}
     {t_Q^{\fe_i-k r+1}} d t_P
\right] d t_Q =- 2\pi \ii.
$$
\end{lemma}

\begin{lemma}\label{2lm:Deltair2}
For $i \in \ZZ_r$ and $k\in \NN$ such that 
$\hH^{(g)}\ni d_h-\fe_i + r k$,
there are $\hphi_j$ $(j<g)$ and
$\hphi_{j'}$ $(j'\ge g)$ such that
$$
\varphi_\infty\left(
\frac{\hphi_j d x}{h_X}\right)
=-t^{d_h -(\fe_{i}+(k-1)r)-r-1}d t(1+d_{>0}(t)), 
$$
$$
\varphi_\infty\left(
\frac{\hphi_{j'} d x}{h_X}\right)
    = -\frac{d t}{t^{\fe_i-k r+1}}(1+d_{>0}(t))
    = -\frac{t^{d_h}d t}{t^{d_h+\fe_i-(k+1)r+(r+1)}}(1+d_{>0}(t)).
$$
\end{lemma}

\begin{proof}
Lemma \ref{2lm:hH^g} shows the existence.
\qed
\end{proof}

\subsubsection{W-normalized differentials of the second kind}
\label{2sssec:WND2nd}

We introduce the W-normalized differentials of the second kind using this $d_{Q} \Sigma\big(P, Q\big)$.

\begin{definition}\label{2df:nu_nu'}
We consider a sufficiently small closed contour $C_\varepsilon$ at $\infty$.
Let $D_\varepsilon$ be the inner side of $C_\varepsilon=\partial D_\varepsilon$ including $\infty$ and $\varepsilon'$ be a point in $D_\varepsilon$ such that $\varepsilon'\neq \infty$.
For differentials $\nu$ and $\nu'$ in $\bH^0(X, \cA(*\infty))$, we define a pairing:
$$
\langle \nu, \nu'\rangle :=\lim_{\varepsilon' \to \infty} \frac{1}{2\pi\ii} \oint_{C_\varepsilon} \left(\int^P_{\varepsilon'}\nu(Q)\right) \nu'(P).
$$
\end{definition}

The following is obtained from the primitive investigation of complex analysis on a compact Riemann surface \cite{FarkasKra}.
\begin{lemma}\label{2lm:nu_nu'}
$\langle \nu, \nu'\rangle=-\langle \nu', \nu \rangle$, and 
$\langle \nu, \nu'\rangle$ does not depend on $\varepsilon$.
\end{lemma}

\begin{definition}\label{2df:nuII}
We define the pre-normalized differentials $\tnuII{i}\in \bH^0(X, \cA_X(*\infty))$  of the second kind $(i=1, 2, \ldots, g)$, which satisfies the relations (if they exist)
\begin{equation}
   \langle \nuI{i}, \tnuII{j}\rangle =\delta_{i j}, \quad
\langle \tnuII{i}, \tnuII{j}\rangle =0.
\label{2eq:nuI_nuII}
\end{equation}
\end{definition}

It is obvious that from Lemma \ref{2lm:nuI_hX} 2 and Lemma \ref{2lm:Deltair2}, $\tnuII{j}$ with $j'$ in Lemma \ref{2lm:Deltair2} is expressed like
$\displaystyle{
\tnuII{j} = \frac{\displaystyle{\sum_{i=0}^{j'} a_{j,i} \hphi_i d x}}
{h_X(x, y_\bullet)}
}$.
\medskip

Noting Proposition \ref{2pr:dSigma1} 3.b, we have the following relations:

\begin{corollary}\label{2cr:nuII_g}
$\displaystyle{
\tnuII{g}=\frac{\hphi_{g}}{h_X(x,y_\bullet)}d x}$,
$\wt\, \tnuII{g}=-2$, and 
$\wt\, \tnuII{1}=-c_X$.
\end{corollary}

\begin{proof}
$-\wt(\hphi_{g})+\wt(\hphi_{g-1})=2$ and thus
$$
\frac{1}{2\pi\ii}
\oint 
\left(
\int^{(x,y_\bullet)}
 \frac{\hphi_{g-1}}{h_X(x',y'_\bullet)}d x'\right)
\frac{\hphi_{g}}{h_X(x,y_\bullet)}d x = 2\pi \ii.
$$
On the other hand,
$-\wt(\hphi_{0})=d_h - r +1 -c_X$ in Lemma \ref{2lm:nuIo2} means
$$
\wt\ \nuI{1}=\wt 
\left(\frac{\hphi_0 d x}{h_X(x, y_\bullet)}\right) 
=c_X-2.
$$
Hence $\wt\, \tnuII{1}=-c_X$ and we obtain 
$$
\tnuII{1}=\frac{\hphi_{0}^* d x}{h_X}, \quad
\hphi_0^*=\frac{1}{t^{d_h+c_X-r+1}}(1+d_{>0}(t)), \quad
 d x =\frac{d t}{t^{r+1}}(1+d_{>0}(t)).
$$
\qed
\end{proof}

\begin{theorem} \label{2th:dSigma}
There exist the pre-normalized differentials $\nuII{j}$ $(j=1, 2, \cdots, g)$ of the second kind such that they have a simple pole at $\infty$ and satisfy the relation,
\begin{equation}
d_{Q} \Sigma\big(P, Q\big) -
  d_{P} \Sigma\big(Q, P\big)=
     \sum_{i = 1}^{g} \Bigr(
         \nuI{i}(Q)\otimes \nuII{i}(P)
        - \nuI{i}(P)\otimes \nuII{i}(Q)
     \Bigr)
   \label{2eq:dSigma},
\end{equation}
where the set of differentials $\{\nuII{1}$, $\nuII{2}$, $\nuII{3}$, $\cdots$, $\nuII{g}\}$ is determined modulo the linear space spanned by $\langle\nuI{j}\rangle_{j=1, \ldots, g}$ and $d R_X$.

We call these $\nuII{i}$'s \emph{W-normalized differentials of the second kind}.\footnote{
In \^Onishi's articles and our previous articles
[\cite{KMP13, MP08, MP14, Onishi05, Onishi18} and references therein],
the definition of W-normalized differentials $\nuII{}$ of the second kind differs from this definition by its sign.
The difference is not significant but has the effect on the Legendre relation in Proposition \ref{2pr:L-rel} and the sign of the quadratic form in the definition of sigma function in Definition \ref{3df:sigma_func}.}
\end{theorem}

\begin{proof}
Noting Proposition \ref{2pr:dSigma_nuII},
Lemmas \ref{2lm:OmegaH2} and \ref{2lm:Deltair} show the fact.
\qed
\end{proof}

\begin{theorem}\label{2th:dRhamcoh}
$$
(\bH^0(X, \cA_X(*\infty))/d\bH^0(X, \cO_X(*\infty))=
\bigoplus_{i=1}^g (\CC \nuI{i}\oplus \CC \nuII{i})
$$
and thus $\dim_\CC(\bH^0(X, \cA_X(*\infty))/d\bH^0(X, \cO_X(*\infty))=2g.$
\end{theorem}

\begin{proof}
It is obvious that $d\bH^0(X, \cO_X(*\infty))\subset \bH^0(X, \cA_X(*\infty))$;  $(j: dR_X \hookrightarrow \bH^0(X, \cA_X(*\infty))$).
We recall Lemma \ref{2lm:nuI_hX} 2 and 7.
$\bH^0(X, \cA_X(*\infty))\setminus \bH^0(X, \cA_X)$ is the set of differentials of the second kind. In other words, there is no differential with the first order singularity at $\infty$.
The embedding $j$ is realized by 
$\displaystyle{j: d\phi_i \mapsto \frac{h_X d \phi_i}{h_X}}$.
Let $H_{dR_X}:=\{\wt(j(d\phi_i))\ |\ i \in \NN_0\}$.
Then $H_{dR_X}\subset -\NN$ and $\#((-\NN)\setminus (H_{dR_X}+1)))=g$, whereas $\# H_X^\fc=g$. Thus $\#(\BH_X^\fc\setminus H_{dR_X})=2g$.
Due to the Riemann-Roch theorem, for every $\ell$ in $H_X^\fc$, there is an element $k$ in $(-\NN)\setminus (H_{dR_X}+1)$ such that $\ell+k =0$.
It shows the relations.
\qed
\end{proof}

\begin{lemma}\label{dSigma_holo}
$\displaystyle{d_{P_2} \Sigma(P_1, P_2)
     +\sum_{i = 1}^g \nuI{i}(P_1)\otimes \nuII{i}(P_2)}$ is holomorphic over $X\setminus \{P_2\}$ as a function of $P_1$ and is holomorphic over $X\setminus \{P_1\}$ as a function of $P_2$.
\end{lemma}

\begin{proof}
From Proposition \ref{2pr:dSigma1}, $d_{P_2} \Sigma(P_1, P_2)$ is  holomorphic over $X\setminus \{P_2\}$ as a function of $P_1$ whereas $d_{P_2} \Sigma(P_1, P_2)$ is  holomorphic over $X\setminus (\{P_2\}\cup \{\infty\})$ as a function of $P_2$.
The order of the singularity at $Q=\infty$ is $2g$ and thus, which can be canceled by $\nuII{i}(P_2)$.
Since the numerator of $d_{P_2} \Sigma(P_1, P_2)$ in (\ref{2eq:Kform'}) consists of the elements in $R_X\otimes_\CC R_X d x_1 \otimes d x_2$, from Proposition \ref{2pr:dSigma1}, there is no term whose weight is $-1$ in $d_{P_2} \Sigma(P_1, P_2)$ as a function in $Q$.
Noting the homogeneous property of $\tlh_X$ from Lemma \ref{2pr:p_H}, we have the result. 
\qed
\end{proof}


\begin{theorem}
\label{2th:Sigma}
\begin{enumerate}
\item
The one-form,
$$
\nuIII{P_1,P_2}(P):= \Sigma(P, P_1) -  \Sigma(P, P_2),
$$
is a differential of the third kind whose only (first-order) poles are
$P=P_1$ and $P=P_2$, with residues $+1$ and $-1$ respectively.

\item
The symmetric two-form,
\begin{equation}
\Omega(P_1, P_2) := d_{P_2} \Sigma(P_1, P_2)
     +\sum_{i = 1}^g \nuI{i}(P_1)\otimes \nuII{i}(P_2),
\label{2eq:Omega_dSigma}
\end{equation}
is the fundamental differential of the second kind, which has the properties:
\begin{enumerate}
\item $\Omega(P, Q)=\Omega(Q, P)$,

\item for any $\zeta \in G_{\varpi_r(P)}$, 
$\Omega(\zeta P, \zeta Q)=\Omega(P, Q)$  if $\varpi_r(P)=\varpi_r(Q)$,

\item 
 $\Omega(P, Q)$ is holomorphic except $Q$ as a function of $P$ and behaves like 
$$
\Omega(P, Q)=\frac{d t_P d t_Q}{(t_P-t_Q)^2}(1 + d_{>0}(t_P, t_Q)),
$$

\item 
$$
         \int^{P_1}_{P_2} \nuIII{Q_1, Q_2} = 
                \int^{P_1}_{P_2} \int^{Q_1}_{Q_2} 
                 \Omega(P,Q).
$$

\end{enumerate}
\end{enumerate}
\end{theorem}

\begin{proof}
{\it{1}} is directly obtained by Proposition \ref{2pr:Sigma}.
{\it{2}} is obvious from (\ref{2eq:dSigma}) and Proposition \ref{2pr:dSigma1}.
\qed
\end{proof}

We note the W-normalized differentials of the first kind and the second kind,  in Definitions \ref{2df:nuI_hX} and  \ref{2df:nuII}.
\begin{definition}\label{2df:WnuIII}
\begin{enumerate}

\item $\nuIII{P_1,P_2}$ is called the W-normalized differential of the third kind

\item 
$\Omega(P, Q)$ is the W-normalized fundamental differential of the second kind and when it is expressed by
\begin{equation}
\Omega(P_1, P_2) =\frac{F_\Omega(P, Q)d x_P \otimes d x_Q}
{(x_P-x_Q)^2 h_X(x_{P},y_{\bullet P})h_X(x_{Q},y_{\bullet Q})},
\label{2eq:Kform}
\end{equation}
where $F_\Omega$ is called Klein fundamental form in 
 $R_X \otimes_\CC R_X$.
\end{enumerate}
\end{definition}

\begin{lemma}
\label{2lm:limFphig}
We have
\begin{equation}
\lim_{P_1 \to \infty}
\frac{F_\Omega(P_1, P_2)}{\hphi_{g-1}(P_1)(x_1 - x_2)^2}
 = \hphi_{g}(P_2).
\label{2eq:limFphig}
\end{equation}
\end{lemma}

\begin{proof}
The expression (\ref{2eq:Omega_dSigma}) and Corollary \ref{2cr:nuII_g} give the result.
\qed
\end{proof}

For the connection of these algebraic tools with the sigma functions. we define $\Pi$ 
\begin{equation}
\begin{split}
\Pi^{P_1, P_2}_{Q_1, Q_2}
 &:= \int^{P_1}_{P_2} \int^{Q_1}_{Q_2} \Omega(P, Q) \\
 &= \int^{P_1}_{P_2} (\Sigma(P, Q_1) - \Sigma(P, Q_2))
 +\sum_{i = 1}^g \int^{P_1}_{P_2} \nuI{i}(P)
\int^{Q_1}_{Q_2} \nuII{i}(P).
\label{2eq:Omega}
\end{split}
\end{equation}

It has the properties.
\begin{lemma}\label{2lm:Pi_Wn}
$\Pi^{P_1, P_2}_{Q_1, Q_2}=
-\Pi^{P_2, P_1}_{Q_1, Q_2},\quad
\Pi^{P_1, P_2}_{Q_1, Q_2}=\Pi^{Q_1, Q_2}_{P_1, P_2}$.
\end{lemma}

We have the generalized Legendre relation given as follows:\footnote{The sign of this relation depends on the sign of the W-normalized Abelian differentials of the second kind because there are variant definitions. }
We introduce the homological basis $\{\alpha_i, \beta_i\}_{i=1,\cdots, g}$ of $\bH_0(X, \ZZ)$ satisfying
\begin{equation}
\langle \alpha_{i}, \alpha_{j}\rangle =\langle \beta_{i}, \beta_{j}\rangle=0, \quad\langle \alpha_{i}, \beta_{j}\rangle =\delta_{ij}\quad
 (i,j=1, 2, \ldots, g).
\label{2eq:Intsec_hom}
\end{equation}

\begin{definition}\label{2df:eta'_eta''}
We define the complete Abelian integrals of the first kind and the second kind by
$$
\omega'_{i j}:=\int_{\alpha_j} \nuI{i},\quad
\omega''_{i j}:=\int_{\beta_j} \nuI{i}.\quad
\eta'_{i j}:=\int_{\alpha_j} \nuII{i},\quad
\eta''_{i j}:=\int_{\beta_j} \nuII{i},
$$
and the Jacobian (Jacobi variety) by $J_X:=\CC^g/\Gamma_X$
with  $\kappa_J :\CC^g \to J_X$ and $\Gamma_X:=\langle \omega', \omega''\rangle_\ZZ$.
\end{definition}

\begin{proposition}\label{2pr:dSigma_nuII}
If by using $\Sigma'\big(P, Q\big)$ defined in Proposition \ref{2pr:Sigma} 4, we define $\{\nuII{i}'\}_{i=1, \ldots, g}$, these $\nuII{i}$ and $\nuII{i}'$ give the same period matrices in Definition \ref{2df:eta'_eta''}.
\end{proposition}

\begin{proof}
From the Proposition \ref{2pr:dSigma1} 4 and Theorem \ref{2th:dSigma}, it is obvious.
\qed
\end{proof}

Let $ \nuIIIo{Q_1, Q_2}$ be the normalized Abelian differential of the third kind, i.e., $\displaystyle{\oint_{\alpha_i} \nuIIIo{Q_1, Q_2}=0}$ \cite{FarkasKra}.
The following lemma corresponds to  Corollary 2.6 (ii) in \cite{Fay73}. 
\begin{lemma} \label{2lm:Omega_nuIII}
$$
\Pi^{P_1, P_2}_{Q_1, Q_2}
=\int^{P_1}_{P_2} \, \nuIIIo{Q_1, Q_2}
-\sum_{i=1}^g\sum_{j=1}^g \gamma_{i,j}
\int^{P_1}_{P_2} \nuI{i} \int^{Q_1}_{Q_2} \nuI{j},
$$
where $\gamma =  \eta'\omega^{\prime -1}$ and $\trp \gamma = \gamma$.
\end{lemma}

\begin{proof}
See Proposition 5.1 in \cite{KMP13}.
\qed
\end{proof}

\begin{proposition}\label{2pr:L-rel}
$$
\omega''\trp \eta'-\omega' \trp\eta''  = \frac{\pi}{2}\ii 1_g,
$$
where $1_g$ is the unit $g\times g$ matrix.

The following matrix satisfies the  {\it generalized Legendre relation}:
\begin{equation}
   M := \left[\begin{array}{cc}2\omegap{} & 2\omegapp{} \\
               2\etap{} & 2\etapp{}
     \end{array}\right], \quad
  M\left[\begin{array}{cc} & 1_g \\ -1_g & \end{array}\right]{}^t {M}
  =\frac{2\pi}{\sqrt{-1}}\left[\begin{array}{cc} & 1_g \\ -1_g &
    \end{array}\right].
\label{2eq:L-rel}
\end{equation}
\end{proposition}

\begin{proof}
It is the same as Proposition 5.1 in \cite{KMP13}.\qed
\end{proof}

From Definitions \ref{2df:nu_nu'} and \ref{2df:nuII}, we have the following corollary, which is the dual of the homological relations (\ref{2eq:Intsec_hom}):
\begin{corollary}\label{2cor:nuI_nuII}
$\langle \nuI{i}, \nuI{j}\rangle =\langle \nuII{i}, \nuII{j}\rangle=0$,
$\langle \nuI{i}, \nuII{j}\rangle =\delta_{ij}$ for $i,j=1, 2, \ldots, g$.
\end{corollary}

The Galois action on the basis of the Homology $\bH_1(X, \ZZ)$ shows the actions of these period matrices $(\omega', \omega'')$ geometrically:

\begin{lemma}\label{2lm:GaloisAction_eta}
If $X$ is the Galois covering on $\PP$, for the Galois action $\hzeta \in G_X$, i.e., $\hzeta: X \to X$, its associated element $\rho_\hzeta$ of $\Sp(2g, \ZZ)$ acts on $(\omega', \omega'')$ and $(\eta', \eta'')$ by
$$
\hzeta(\omega', \omega'')=(\omega', \omega'')
\trp\rho_\hzeta, \quad
\hzeta(\eta', \eta'')=(\eta', \eta'')
\trp\rho_\hzeta,
$$
and the generalized Legendre relation (\ref{2eq:L-rel}) is invariant for the action.
\end{lemma}

\begin{proof}
Due to the definition of $\Sp(2g, \ZZ)$, we have
$\displaystyle{
\trp\rho_\hzeta
\begin{pmatrix} & -1_g \\ 1_g & \end{pmatrix}\trp
\rho_\hzeta
=\begin{pmatrix} & -1_g \\ 1_g & \end{pmatrix}
}$ and thus, (\ref{2eq:L-rel}) is invariant.
\end{proof}

\section{Sigma functions for W-curves}\label{sc:SigmaFW}

\subsection{W-normalized shifted Abelian integrals}
\label{2ssc:ShiftedAI}
Since the non-symmetric W-curves have the non-trivial $R_X$-module, $\hR_X$ in Definition \ref{2df:nuI_hX}, (and properties in Proposition \ref{2pr:cKX}), the Abel-Jacobi map \cite{FarkasKra} for $\hR_X$ naturally appears.
Let $\tX$ be the Abelian universal covering of $X$, which is  constructed by the path space of $X$; $\kappa_X: \tX \to X$ with $\iota_X: X\to \tX$.
Thus recalling Definitions \ref{2df:fKX_fkX}, and \ref{2df:fK_s}, we introduce the shifted Abelian integral $\tw_\fs$ and the Abel-Jacobi map $w_\fs$ as an extension of the W-normalized Abelian integral $\tw: S^k\tX \to \CC^g$, 
$\displaystyle{\left(\tw(\gamma_1, \cdots, \gamma_k)=\sum_{i} \int_{\gamma_i} \nuI{}\right)}$, $\displaystyle{\tw^\circ:=\frac{1}{2}\omega^{\prime-1}\tw}$, and the Abel-Jacobi map $w: S^k X \to J_X$,
$\displaystyle{\left(w(P_1, \cdots, P_k)=\sum_{i} \int_\infty^{P_i} \nuI{}\right)}$, $\displaystyle{w^\circ:=\frac{1}{2}\omega^{\prime-1}w}$ as mentioned in \cite{KMP16}:

\begin{definition}\label{2df:sAbl_int}
We define $\tw_\fs$ and  $w_\fs$ by
$$
   \tw_\fs:S^k\tX \to \CC^g, \quad 
 \tw_\fs(\gamma_1, \cdots, \gamma_k) :=
 \left(\sum_{i=1}^k \tw(\gamma_i)\right)+ \tw(\iota_X \fK_\fs) =
 \sum_{i=1}^k \int_{\gamma_i} \nuI{}+\int_{\iota_X\fK_\fs} \nuI{},
$$
$$
   w_\fs := \kappa_J \circ\tw_\fs\circ \iota_X :S^k X\to J_X,
$$
and $\tw_\fs^\circ$ and  $w_\fs^\circ$ by
$$
\tw_\fs^\circ:=\frac{1}{2}\omega^{\prime-1}\tw_\fs:S^k\tX \to \CC^g, \quad
w_\fs^\circ:=\frac{1}{2}\omega^{\prime-1}w_\fs:
S^k X \to J_X^\circ:=\CC^g/\Gamma_X^\circ,\quad
\Gamma_X^\circ:=\omega^{\prime -1} \Gamma_X.
$$
\end{definition}

For symmetric numerical semigroup case $H_X$, $\tw_\fs = \tw$, $\tw_\fs^\circ = \tw^\circ$, $w_\fs = w$, and $w_\fs^\circ = w^\circ$.

For given a divisor $D$, let $|D|$ be the set of effective divisors linearly equivalent to $D$ and identified with $\bH^0(X, \cO_X(D))$.

\begin{definition}\label{2df:W_X^k}
We define the Wirtinger varieties by
$$
W_X^k:=w(S^k X), \ W_{X,\fs}^k:=w_\fs(S^k X)\subset J_X, \quad 
W_X^{\circ k}:=w^\circ(S^k X),\ 
W_{X,\fs}^{\circ k}:=w_\fs^\circ(S^k X)\subset J_X^\circ,
$$
and their strata,
$$
W_X^{k,1}:=w(S^k_1 X),\quad
W_{X,\fs}^{k,1}:=w_\fs(S^k_1 X), \quad 
W_X^{k;1}:=w^\circ(S^k_1 X),\quad
W_{X,\fs}^{k;1}:=w_\fs^\circ(S^k_1 X),
$$
where 
$
S^n_m X := \{D \in S^n X  \ | \
    \dim | D | \ge m\}.
$
\end{definition}

\bigskip

We would encounter the several results which are obtained via the embedding $\iota_X : S^k X \to S^k \tX$. 
For such cases, we sometimes omit $\iota_X$ for maps $\tw\circ \iota_X$ and $\tw_\fs\circ \iota_X$, e.g., for $\tw\circ\iota_X : S^k X \to \CC^g$, we simply write $\tw(P_1, P_2, \ldots, P_k)$ rather than $\tw(\iota_X(P_1, P_2, \ldots, P_k))$.

\subsection{Riemann theta function of W-curves}\label{ssc:Rtheta}

The Riemann theta function, analytic in both variables $z\in \CC^g$ and $\tau:=\omega^{\prime -1}\omega''$, is defined by
\begin{equation}
\theta(z,\tau ) =
\sum_{n \in \ZZ^{g}} 
\exp\left( 2\pi\ii ({}^t n z + \frac{1}{2}{}^t n \tau n)\right).
\label{3eq:theta2.0}
\end{equation}
(For a given W-curve $X$, we simply write it as $\theta(z)=\theta(z,\tau )$ by assuming that the Homology  basis is implicitly fixed.)
By letting $\Gamma_X^\circ:=\langle 1_g, \tau\rangle_\ZZ$ and $J_X^\circ:=\CC^g/\Gamma_X^\circ$, the zero-divisor of $\theta$ modulo $\Gamma^\circ_X$ is denoted by  $\Theta^\circ_X:=
\kappa_J\  \mathrm{div}(\theta) \subset J_X^\circ$.

The $\theta$ function with characteristic $\delta', \delta''\in\RR^{g}$ is defined as:\footnote{There is another definition, e.g, in \cite{Mum81, Mum84}, in which $\delta''$ and $\delta'$ are exchanged in our definition.}
\begin{equation}
\theta \left[\begin{matrix}\delta''\\ \delta'\end{matrix}
\right] (z, \tau )
   =
   \sum_{n \in \ZZ^g} \exp \big[\pi \sqrt{-1}\big\{
    \ ^t\negthinspace (n+\delta'')
      \tau(n+\delta'')
   + 2\ {}^t\negthinspace (n+\delta'')
      (z+\delta')\big\}\big].
\label{3eq:theta2.1}
\end{equation}
We also basically write it as $\theta[\delta](z)=\theta[\delta](z,\tau )$.

The shifted Abelian integral $\tw_\fs^\circ$
and the Abel-Jacobi map $w_\fs^\circ$ in Definition \ref{2df:sAbl_int} lead
to the shifted Riemann constant \cite{KMP16}:
\begin{proposition}\label{3pr:SRieConst}

\begin{enumerate}
\item
If $\fk_X$ in Definition \ref{2df:fKX_fkX} is not zero or $X$ is not
symmetric, the Riemann constant $\xi_X$ is not a half period of $\Gamma_X^\circ$.

\item
The shifted Riemann constant $\displaystyle{\xi_{X,\fs}:=\xi_X-\tw_\fs^\circ(\iota_X\ \fK_\fs)}$ for every W-curve $X$ is the half period of $\Gamma_X^\circ$.

\item
By using the shifted Abel-Jacobi map, we have
$$
\Theta_X^\circ = w_\fs^\circ(S^{g-1} X)+\xi_{X,\fs}
\quad \mbox{modulo}\quad \Gamma_X^\circ,
$$
i.e., for $P_i \in X$,
$
\theta\left(\tw_\fs^\circ\circ\iota_X(P_1, \ldots, P_{g-1})
+\xi_{X,\fs}\right)=0
$ for every W-curve $X$.

\item The following holds
$$
\Theta_X^\circ = w_\fs^\circ(S^{g-1} X)+\xi_{X,\fs}
=w^\circ(S^{g-1} X)+\xi_{X}
\quad \mbox{modulo}\quad \Gamma_X^\circ,
$$

\item
There is a $\theta$-characteristic $\delta_X$ of a half period which 
represents the shifted Riemann constant $\xi_{X,\fs}$, i.e.,
$
\theta[\delta_X]
\left(\frac{1}{2}\omega^{\prime-1}\tw_\fs\circ\iota_X(P_1, \ldots, P_{g-1}) \right)=0
$, i.e., 
$\displaystyle{
\delta_X=
\left[
\begin{matrix}
\delta_X''\\
\delta_X'
\end{matrix}\right]}$, 
$\xi_{X,\fs} \equiv \delta_X'+ \tau_X \delta_X''$ modulo $\Gamma_X^\circ$.

\end{enumerate}

\end{proposition}

The following comes from the investigation of the truncated Young diagram and the Schur polynomials in \cite{MP14}; though we did not consider the Young diagram associated to plane curve in the paper \cite{MP14}, the investigation is easily generalized to general Young diagrams associated with any numerical semigroups (c.f. Lemma \ref{2lm:Lambda^k}).
Thus we state the facts without proofs.

\begin{proposition}\label{3pr:L^k}
For the Young diagram $\Lambda_X$ associated with the numerical semigroup $H_X$ of genus $g$, an integer $k$ $(0\le k < g)$, and the characteristics of the partition of $\Lambda_X^{[k]}$ $=(\Lambda_{k+1}, \Lambda_{k+2}, \ldots, \Lambda_{g})$ $=(a_{n_k}, \ldots, a_2, a_1; b_{n_k}, \ldots, b_2, b_1)$, ($n_k := r_{\Lambda_X^{[k]}}$, the rank of $r_{\Lambda_X^{[k]}}$), the following holds:
\begin{enumerate}
\item $N^\fc(g-k-i)-N(i-1)$ $(i=1, 2, \ldots, n_k)$
is an element of the gap sequence 
$H_X^\fc$, and thus let 
$N^\fc(L^{[k]}_i):=N^\fc(g-k-i)-N(i-1)$ and then we have
$$
	\Lambda_{g-L^{[k]}_i} + g - L^{[k]}_i = a_{i} + b_{i} + 1
$$
for every $i =  1, \ldots, n_k$, and

\item $L^{[k]}_{1} = k + 1$.

\end{enumerate}
\end{proposition}

\begin{definition}\label{3df:I_Xk}
Let $\mathrm{Index}(g,\ell):=\{1, \ldots, g\}^\ell$.
We define the sequences $\natural_{X,k}$ (simply $\natural_{k}$) and 
$\natural_{X,k}^{(i)}$ (simply $\natural_{k}^{(i)}$) as elements in $\mathrm{Index}(g,n_k)$ given by
$$
\natural_k=
\natural_{X,k} :=\{ L^{[k]}_1, L^{[k]}_2,
\ldots, L^{[k]}_{n_k}\},\qquad
\natural_k^{(i)}=
\natural_{X,k}^{(i)}:=(\natural_{X,k}\setminus\{k+1\})\bigcup \{i\}.
$$
\end{definition}

Let us consider $\CC[\fu_1, \ldots, \fu_g]$ and the symmetric polynomials, e.g., the power symmetric polynomials,
$\displaystyle{
T^{(\ell)}_k :=\frac{1}{k}\sum_{i=1}^\ell \fu_i^k
}$ for $\fu=\trp(\fu_1,\ldots, \fu_g)$, and 
$T_k = T^{(g)}_k$.

For a Young diagram $\Lambda=(\Lambda_1, \Lambda_2, \ldots, \Lambda_n)$, the Schur function $s_\Lambda$ is defined by the ratio of determinants of $n\times n$ matrices \cite{MP14},
$$
	s_{\Lambda}(\fu) = \frac{|\fu_{i}^{\Lambda_j+n-j}|}{|\fu_i^{j-1}|}.
$$
When $\Lambda$ associated with the semigroup $H$ as in Subsection \ref{2sc:WSG}, it can be also regarded as a function of
$\displaystyle{
          T= \trp(T_{\Lambda_1+n-1},\ldots,T_{\Lambda_n})}$ \cite{MP14}, 
and thus, we express it by
$$
	S_{\Lambda}(T) = s_{\Lambda}(\fu).
$$

\bigskip
We recall the truncated Young diagrams $\Lambda_X^{(k)}$ and $\Lambda_X^{[k]}$ for the Young diagram $\Lambda_X$ associated with the W-curve $X$ in Definition \ref{2df:Lambda^k}.
We define  $\fu^{[k]}=(\fu_{1}^{[k]}, \ldots, \fu_g^{[k]})$. where $\fu_i^{[k]} := T^{(k)}_{\Lambda_i+g-i}$ and let
$$
\bs_{\Lambda_X^{(k)}}(\fu^{[k]}) := 
S_{\Lambda_X^{(k)}}(T^{(k)})|_{T^{(k)}_{\Lambda_i+g-i} = u_i^{[k]}}.
$$
We also write the decomposition, $\fu^{[g]} = \fu^{[k]} + \fu^{[g;k]}\in \CC^g$.

\begin{proposition}\label{3pr:Schur_dSchur}
For $\natural_{X,k}$ in Definition \ref{3df:I_Xk},
\begin{equation}
\bs_{\Lambda^{(k)}_X}(\fu^{[k]}) = 
\varepsilon_{\Lambda,\natural_{X,k}}
\left(\prod_{i\in \natural_{X,k}}\frac{\partial}{\partial \fu_i^{[g]}}\right)
\bs_{\Lambda^{(g)}_X}(\fu^{[g]})
\Bigr|_{\fu^{[g]}=\fu^{[k]}}.
\label{3eq:bs_bs_gk}
\end{equation}
\end{proposition}

Following Nakayashiki's results in \cite{Nak16}, we state the Riemann-Kempf theorem \cite{ACGH85, Onishi05} of the W-curves.

\begin{proposition}\label{3th:RieKempf_theta}
For $u\in \CC^g$ and a multi-index $J\subset 
\mathrm{Index}(g, \ell)$, we define
$\displaystyle{
\partial_I:=\prod_{j \in J}\frac{\partial}{\partial u_{j}}.
}$
Let $\Lambda_X$ be the Young diagram of the W-curve $X$ and for given $k \in \{0, 1, \ldots, g\}$, let $\Lambda_X = \Lambda_X^{[k]} \cup \Lambda_X^{(k)}$.
For every multi-index $I=\{\alpha_1, \ldots, \alpha_m\} \in\mathrm{Index}(g,m)$, $m < n_k=r_{\Lambda_X^{[k]}}$, $N_k:=|\Lambda_X^{[k]}|$, and $u \in \Theta_X^k$, 
\begin{enumerate}
\item
$
\partial_I \theta\left( (2\omega^{\prime})^{-1} u+ \xi_X,\tau\right) = 0,
$
whereas
$
\partial_{\natural_{X,k}}
 \theta\left( (2\omega^{\prime})^{-1} u+ \xi_X,\tau\right) \neq 0,
$

\item for $\ell < N_k$
$
\partial_{u_g}^\ell \theta\left( (2\omega^{\prime})^{-1} u+ \xi_X,\tau\right) = 0,
$
whereas
$
\partial_{u_g}^{N_k}
 \theta\left( (2\omega^{\prime})^{-1} u+ \xi_X,\tau\right) 
\neq 0.
$
\end{enumerate}
\end{proposition}

\begin{proof} See Corollary 3 in \cite{Nak16}.
\end{proof}

\subsection{Sigma function and W-curves}\label{ssc:SigmaF}

We now define the sigma function following Nakayashiki \cite[Definition 9]{Nak16}. We remark that due to the shifted Riemann constant, our definition differs from Nakayashiki's so that our sigma function has the natural properties, including the parity and Galois action and the fact that the point of expansion by Schur polynomials is also shifted as mentioned in Theorem \ref{thm:sigma}.
In other words, we employ some parts of the definition of the sigma function by Korotkin and Shramchenko \cite{KorotkinS} who defined the several sigma functions with spin structures based on Klein's transcendental approaches.

\begin{definition}\label{3df:sigma_func}
We define $\sigma$ as an entire function of (a column-vector) $u={}^t\negthinspace (u_1, u_2, $\break $\ldots, u_g) \in \mathbb{C}^g$,
\begin{equation}
\begin{split}
   \sigma{}(u) &=\sigma{}(u;M) 
   =\frac{\varepsilon_{\Lambda,\natural_{X,0}}\exp(\tfrac{1}{2}{}\ ^t\negthinspace
u\etap{}{\omegap{}}^{-1}u)
   \vartheta\negthinspace
   \left[\begin{matrix}\delta''\\ \delta'\end{matrix}\right](\frac{1}{2}{\omegap{}}^{-1} u;\
{\omegap{}}^{-1}\omegapp{}) 
}{
\partial_{\natural_{X,0}}\vartheta\negthinspace
   \left[\begin{matrix}\delta''\\ \delta'\end{matrix}\right](\frac{1}{2}{\omegap{}}^{-1} u; {\omegap{}}^{-1}\omegapp{})\Bigr|_{u=\tw(\iota_X \fK_\fs)}}
\\
&=\frac{\varepsilon_{\Lambda,\natural_{X,0}}\exp(\tfrac{1}{2}\ ^t\negthinspace
u\etap{}{\omegap{}}^{-1}\  u)}
{\partial_{\natural_{X,0}}\vartheta\negthinspace
   \left[\begin{matrix}\delta''\\ \delta'\end{matrix}\right]
(\frac{1}{2}{\omegap{}}^{-1} u; {\omegap{}}^{-1}\omegapp{})|_{u=\tw(\iota_X \fK_\fs)}} \\
   &\hskip 20pt\times
   \sum_{n \in \ZZ^g} \exp \big[\pi \sqrt{-1}\big\{
    \ ^t\negthinspace (n+\delta'')
      {\omegap{}}^{-1}\omegapp{}(n+\delta'')
   + \ ^t\negthinspace (n+\delta'')
      ({\omegap{}}^{-1} u+2\delta')\big\}\big],
\end{split}
   \label{3eq:de_sigma}
\end{equation}
where  $\varepsilon_{\Lambda,\natural_{X,k}}$ is defined in 
(\ref{3eq:bs_bs_gk}) and $\natural_{X,0}$ is defined in
Definition \ref{3df:I_Xk}.
\end{definition}

Then we have the following theorem, i.e., Theorem \ref{thm:sigma}.

It is worthwhile noting that the following (\ref{3eq:RJFR}) obviously leads the Jacobi inversion formulae on the Jacobian $J_X$ and its strata as mentioned in \cite{KMP13, KMP19, MP08,MP14}; though we omit the inversion formulae for the reason of space, we can easily obtain them as its corollary following \cite{MP08,MP14}.
Since $\Pi_{P_i, P'_j}^{P, Q}$ in (\ref{3eq:RJFR}) can be expressed in terms of $R_X$ in (\ref{2eq:Omega}), we can represent the elements of $R_X$ by using the differentials of the sigma functions.
More explicitly, since the Jacobi inversion formulae on $J_X$ provide that the multi-variable differentials of the sigma are equal to the meromorphic functions of $\hR_X$ as predicted in \cite{KM2020}, they imply that if the formulae are integrable, the sigma function is, in principle, obtained by integrating the meromorphic functions on $S^g X$; since the integrability is obvious, the sigma function for every W-curve can be, in principle, algebraically obtained like the elliptic sigma function in Weierstrass' elliptic function theory.
They also show the equivalence between the algebraic and transcendental properties of the meromorphic functions on $X$.
The sigma function is defined for every compact Riemann surface by Nakayashiki following Klein's construction of his sigma functions \cite{Klein}. Klein defined his sigma functions using only the data of hyperelliptic Riemann surfaces, following Riemann's approach.
On the other hand, Weierstrass criticized Riemann's approach and insisted on the algebraic ways, associated with Weierstrass curves.\footnote{see Weierstrass's words in a letter to Schwarz (Werke II, 235) cited by Poincare \cite{Poincare10}:{\lq\lq}Plus je r\'efl\'echis aux principes de la th\'eorie des fonctions - et c'est ce que je fais sans cesse - plus je suis solidement convaincu qu'ils sont b\^atis sur le fondement des v\'erit\'es alg\'ebriques et que, par cons\'equent, ce n'est pas le v\'eritable chemin, si inverse ment ou fait appel au transcendant pour \'etablir les th\'eor\`emes simples et fondamentaux de l'Alg\'ebre ; et cela reste vrai, quelque p\'en\'etrantes que puissent para\^itre au premier abord les consid\'erations par lesquelles Riemann a d\'ecouvert tant d'importantes propri\'et\'es des fonctions alg\'ebriques.{\rq\rq}({\lq\lq}The more I think about the principles of function theory -- and I do continuously - the more I am convinced that this must be built on the foundations of algebraic truths [my emphasis], and that it is consequently not correct to resort on the contrary to the {\lq}transcendent{\rq}, to express myself briefly, as the basis of simple and fundamental algebraic propositions.
This view seems so attractive at first sight, in that through it Riemann was able to discover so many important properties of algebraic functions.{\rq\rq}\cite{Jah})}
Unifying Klein's and Weierstrass' views, Baker reformulated Klein's sigma functions after defining explicit algebra curves, and connected the sigma functions and  the meromorphic functions of the curves like Weierstrass' elliptic function theory \cite{Baker97}.
Thus we emphasize that the following theorem implies completing Weierstrass' program by succeeding Baker's approaches.

\begin{theorem}\label{thm:sigma}

$\sigma(u)$ has the following properties:

\begin{enumerate}

\item it is modular invariant,

\item it obeys the translational formula;
for $u$, $v\in\CC^g$, and $\ell$ $(=2\omegap{}\ell'+2\omegapp{}\ell'')$ $\in\Gamma_X)$, if we define
$
  L(u,v) :=2\ {}^t{u}(\etap{}v'+\etapp{}v''),\
  \chi(\ell):=\exp[\pi\sqrt{-1}\big(2({}^t {\ell'}\delta''-{}^t
  {\ell''}\delta') +{}^t {\ell'}\ell''\big)],$
the following holds
\begin{equation}
	\sigma{}(u + \ell) =
\sigma{}(u) \exp(L(u+\frac{1}{2}\ell, \ell)) \chi(\ell),
        \label{3eq:pperiod}
\end{equation}

\item its divisor is $\kappa_J^{-1} \Theta_{X,\fs} \subset \CC^g$, where $ \Theta_{X,\fs}:=w_\fs(S^{g-1} X) \subset J_X$,

\item it satisfies the Jacobi-Riemann fundamental relation,
For $(P, Q, P_i, P'_i) \in X^2 \times (S^g(X)\setminus S^g_1(X)) \times 
(S^g(X)\setminus S^g_1(X))$, 
$$
	u  := \tw_\fs(P_1, \ldots, P_g), \quad
	v  := \tw_\fs(P'_1, \ldots, P'_g),  \quad
$$
\begin{equation}
\begin{split}
\exp\left( 
\sum_{i, j = 1}^g 
   \Pi_{P_i, P'_j}^{P, Q} \right)
&=
\frac{\sigma(\tw(P) - u) \sigma(\tw(Q) - v)}
     {\sigma(\tw(Q) - u) \sigma(\tw(P) - v)}\\
&=\frac{\sigma(\tw(P) - \tw_\fs(P_1, \cdots, P_g)) 
        \sigma(\tw(Q) - \tw_\fs(P'_1, \cdots, P'_g))}
     {\sigma((\tw(Q) - \tw_\fs(P_1, \cdots, P_g)) 
      \sigma(\tw(P) - \tw_\fs(P'_1, \cdots, P'_g))},
\end{split}
\label{3eq:RJFR}
\end{equation}
which generates the Jacobi inversion formulae for $S^k X$,

\item the leading term in the Taylor expansion of the $\sigma$ function associated with $X$, with normalized constant factor $c$, is expressed by the Schur function of $\Lambda_X$
$$
   \sigma{}(u+\tw(\iota_X \fK_\fs)) 
= S_{{\Lambda_X}}(T)|_{T_{\Lambda_i + g - i} = u_i}
            + \sum_{|\rw_g(\alpha)|>|\Lambda_X|} a_\alpha u^\alpha,
$$
where $a_\alpha\in \CC[\lambda_{ij}]$,
$\alpha = (\alpha_1,..., \alpha_g)$,
 $u^\alpha = u_1^{\alpha_1} \cdots u_g^{\alpha_g}$, 
and 
$\displaystyle{\rw_g(\alpha)=\sum\alpha_i\wt(u_i)}$.
Here $S_\Lambda (T)$ is the lowest-order
term in the w-degree of the $u_i$; 
$\sigma(u)$ is homogeneous of degree  $|\Lambda_X|$ 
 with respect to $\wt_\lambda$,

\item $
\sigma(-u) = \pm \sigma(u)$, and

\item If $\hzeta \in G_X$ satisfies $\hzeta^\ell = \mathrm{id}$, and $\hzeta[ \sigma(u+\tell)/\sigma(u)]=\sigma(u+\tell)/\sigma(u)$ for $\tell \in \Gamma_X$ and $u\in \CC^g$, the action gives the one-dimensional representation such that
$$
\hzeta \sigma(u) = \rho_{\hzeta} \sigma(u),
$$
where $\rho_{\hzeta}^\ell=1$.

\end{enumerate}
\end{theorem}

\begin{proof}
{\it{1}} and {\it{5}} are obtained by Theorem 13 in \cite{Nak16} by noting the difference of the definition of our sigma function in Definition \ref{3df:sigma_func} from Nakayashiki's \cite[Definition 9]{Nak16}.
{\it{3}} is due to Proposition \ref{3pr:SRieConst}.
{\it{2}} is standard and can be obtained by the direct computations \cite{MP08}.
{\it{4}} is the same as Proposition 4.4 in \cite{MP08}.
{\it{6}} and {\it{7}} are the same as Lemma 3.6, and Lemma 4.1 in \cite{Onishi05}.
\qed
\end{proof}

\bigskip

\subsection{Conclusion and Discussions}

We have considered the Weierstrass curves (W-curves), which are algebraic expression of compact Riemann surfaces; The set of W-curves represent the set of compact Riemann surfaces.

By using the algebraic tools we constructed in this paper,we have a connection between the sigma function for W-curve $X$ and the meromorphic functions on $X$ as in Theorem \ref{thm:sigma}.
Since the Jacobi inversion formulae via Theorem \ref{thm:sigma} 4 are given by the differential identity, by integrating it, it, in principle, provides that the sigma function is constructed by an integral formula of the meromorphic functions on the W-curve $X$. 
In other words, we give an algebraic construction of the sigma function, or so-called the EEL-construction \cite{EEL00} in this paper.

It is noted that this construction is based on our recent result on the trace structure of the affine ring $R_X$ \cite{KMP2022a}.

\bigskip

Further we also discuss mathematical meaning of our result as follows.
We also note that for an ordinary point $P$ in every W-curve $Y$ with $H_Y$ at $\infty \in Y$ of genus $g$, the Weierstrass gap sequence at $P$ is given by the numerical semigroup $H^\fc=\{1, 2, \cdots, g\}$ and there is a W-curves $X$ which is birationally equivalent to $X$ such that $\infty\in X$ corresponds to $P\in Y$ and $H_X^\fc=H^\fc$. 
Then there appear two sigma functions $\sigma_Y$ for the W-curve $Y$ with non-trivial Weierstrass semigroup $H_Y$ and $\sigma_X$ for the W-curve $X$ with $H_X=H$.
By some arguments on the both Jacobians $J_X$ and $J_Y$, we find that $\sigma_X$ and $\sigma_Y$, due to the translational formulae and so on, are the same functions, and the both sifted Abelian integrals agree. 
Then the above theorem (5) means that we have the expansions of the sigma function at the Abelian image of the ordinary point $P$ in $Y$.
It means that the problem of finding the expansion of the sigma function for a point $u$ in Jacobian $J_Y$ is reduced to the problem that we should find the birational curves associated with the preimage of the Abelian integral.
Since in Weierstrass' elliptic function theory, we often encounter the reductions of the transcendental problems to the algebraic problems, we also remark that this reduction has the same origin, i.e., the equivalence between algebraic objects and transcendental objects in the Abelian function theory. 

With Theorem \ref{thm:sigma}, we recognize that this theorem is the goal Weierstrass had in mind, and at the same time, with it, we also recognize that we finally reached the starting point for the development of the Weierstrass program to construct an Abelian function theory for every W-curve $X$ like his elliptic function theory.

\end{document}